\documentclass[final, 3p, times]{elsarticle}
\usepackage{amssymb,amsfonts,amsthm, amsmath}
\usepackage[hidelinks]{hyperref}
\usepackage{cleveref}
\usepackage{mathrsfs}
\usepackage{bm}
\usepackage{mathtools}
\usepackage{tikz-cd}
\usepackage{svg}
\usepackage{multirow}

\DeclareMathOperator{\sym}{sym}
\DeclareMathOperator{\asym}{asym}
\DeclareMathOperator{\skw}{skw}

\DeclareMathOperator{\Tr}{Tr}
\let\div\relax
\newcommand{\div}{\nabla \cdot}
\DeclareMathOperator{\divo}{div}

\newcommand{\BDM}[1]{\mathbb{BDM}_{#1}}
\newcommand{\RT}[1]{\mathbb{RT}_{#1}}
\renewcommand{\P}[1]{\mathbb{P}_{#1}}
\renewcommand{\L}[1]{\mathbb{L}_{#1}}

\journal{ }

\begin{document}

\newtheorem{theorem}{Theorem}
\newtheorem{lemma}[theorem]{Lemma}
\newtheorem{remark}[theorem]{Remark}
\newtheorem{definition}[theorem]{Definition}

\numberwithin{equation}{section}
\numberwithin{theorem}{section}
\numberwithin{table}{section}


\begin{frontmatter}

\title{Multipoint stress mixed finite element methods for the linear Cosserat equations}

\author[label1]{Wietse M. Boon}
\author[label2]{Alessio Fumagalli}
\author[label3,label1]{Jan M. Nordbotten}
\author[label4]{Ivan Yotov}

\address[label1]{
  Division of Energy and Technology, NORCE Norwegian Research Centre, Bergen, Norway
            }
\address[label2]{MOX Laboratory, Department of Mathematics, Politecnico di Milano, Milan, Italy}

\address[label3]{Department of Mathematics, University of Bergen, Bergen, Norway} 

\address[label4]{Department of Mathematics, University of Pittsburgh, Pittsburgh, USA}

\begin{abstract}
    We propose mixed finite element methods for Cosserat materials that use suitable quadrature rules to eliminate the Cauchy and coupled stress variables locally. The reduced system consists of only the displacement and rotation variables. Four variants are proposed for which we show stability and convergence using a priori estimates. Numerical experiments verify the theoretical findings and higher order convergence is observed in some variables.
\end{abstract}

\begin{keyword}
    multipoint stress \sep Cosserat \sep hybridization \sep mixed-finite elements

\end{keyword}

\end{frontmatter}

\section{Introduction}

The equations governing Cosserat materials \cite{cosserat1909theorie} form an extension of linearized elasticity that incorporates local rotations in the medium as an independent variable and a non-negative parameter, which we denote by $\ell$, that represents the scale separation. The equations are suitable for modeling micropolar media, such as granular or composite porous media \cite{ehlers2020cosserat}.
These systems are typically solved by using the displacements and rotation as primary variables \cite{RIAHI20093450}. A mixed formulation that includes mechanical and couple stresses was recently analyzed in \cite{boon2025mixed} and investigated numerically in \cite{nordbotten2024mixed}. 
This formulation presents advantages such as avoiding locking phenomena and maintaining robustness in case the system degenerates to the equations for linearized elasticity. However, the inclusion of the two additional stress fields significantly increases the computational cost, particularly for three-dimensional problems. 

To mitigate this additional cost, this work proposes discretization methods inspired by the multipoint flux mixed finite element method \cite{wheeler2006}. This methodology has previously been extended from Darcy flow to elasticity, Stokes flow, and Biot poroelasticity models in \cite{ambartsumyan2020multipoint,msfmfe-Biot2020,Caucao2022,Boon2022b,Boon2022c,egger2020second}.
The key idea is to introduce a low-order quadrature rule which makes the mass matrices associated with the stress variables block diagonal. In turn, these can easily be inverted,  resulting in a Schur complement system that depends only on the displacement and rotation variables. The Cauchy and couple stresses can be post-processed.

We propose four methods based on different choices of mixed finite element spaces and demonstrate both theoretically and numerically that the stability and linear convergence of the methods remain intact after the application of the localized quadrature rule. Moreover, we emphasize that the reduced methods are stable and convergent in the degenerate limit of linearized elasticity $\ell = 0$. A key component in the analysis is the use of a discrete norm for the divergence of the couple stress, which is based on the projection to the rotation finite element space. This allows us to both avoid the assumption in \cite{boon2025mixed} that $\ell$ is a piecewise linear function, and to analyze methods with finite element spaces that were not considered in \cite{boon2025mixed}. In particular, three of the methods are new mixed finite element methods for the Cosserat problem and in some of these cases they result in new (non-reduced or reduced) mixed finite element methods for linear elasticity with weak stress symmetry when $\ell = 0$.

The first method we study is a first-order based on mixed finite element spaces considered in \cite{boon2025mixed}. The second method is a new first-order method for the linear Cosserat system in which the rotation variable can additionally be eliminated when $\ell = 0$, recovering the multipoint stress mixed finite method for linear elasticity developed in \cite{ambartsumyan2020multipoint}. The final two methods are new methods for the linear Cosserat equations based on higher-order finite elements. In the third method, when $\ell = 0$, both the mixed finite element method and its multipoint stress version are new second-order mixed methods for linear elasticity with weak stress symmetry. In the fourth method, when $\ell = 0$, we recover the method proposed in \cite{lee2016towards} in the non-reduced case, whereas the reduced method is a new second-order multipoint stress mixed finite element method for linear elasticity with weak stress symmetry.

The article is organized as follows. In \Cref{sec:model} we present the equations governing linear Cosserat materials. \Cref{sec: MS-MFE} presents the general analysis strategy that we apply to the multipoint stress mixed finite element methods proposed in the four subsequent \Cref{sec:simple_scheme,sec:reducible_scheme,sec:RT1-L1,sec:RT1-P1}.
The performance of these methods is investigated numerically in \Cref{sec:num}. \Cref{sec:conclusion} contains the conclusions.

\subsection{Preliminary definitions and notation}
\label{sub: preliminaries}

Let $d=2$ or 3 be the spatial dimension of our problem and let $\Omega \in \mathbb{R}^d$ be a domain with Lipschitz boundary $\partial \Omega$, having outward unit normal $n$.
We assume that the boundary $\partial \Omega$ is divided into two disjoint parts $\partial_e \Omega$ and $\partial_n \Omega$ on which essential and natural boundary conditions are imposed, respectively.

Let $\mathbb{M} \coloneqq \{ \tau: \Omega \to \mathbb{R}^{d \times d} \}$ be the space of matrix-valued functions,
let $\mathbb{V} \coloneqq \{u: \Omega \to \mathbb{R}^d\}$ be the space of vector-valued functions, and
let $\mathbb{K} \coloneqq \{r: \Omega \to
\mathbb{R}^{k_d}\}$, with $k_d = \binom{d}{2}$, be the space of functions that are 
vector-valued in 3D and scalar-valued in 2D. Finally, we set $\mathbb{W} \coloneqq \{\omega: \Omega \to \mathbb{R}^{d \times k_d}\}$, which coincides with $\mathbb{M} $ if $d=3$ and with $\mathbb{V}$ if $d=2$.
%
For a given function space $\mathbb{X}$, let $L^2 \mathbb{X}$ be the space
of square-integrable functions in $\mathbb{X}$, endowed with a scalar product $(\phi, \psi)_\Omega =\int_\Omega \phi \psi$ and induced
norm $\| \phi \|_\Omega\eqqcolon \sqrt{(\phi, \phi)_\Omega}$. Let $H_{\divo} \mathbb{X} \subset L^2 \mathbb{X}$ be the subspace that contains functions with square-integrable divergence.
An apostrophe on a function space indicates its dual space whereas apostrophes on functions denote test functions. Angled brackets denote duality pairings.

For $\tau \in \mathbb{M}$, we define the following algebraic operators:
the trace operator $\Tr: \mathbb{M} \to \mathbb{R}$ as $\Tr \tau = \sum_i \tau_{ii}$,
the symmetry operator $\sym: \mathbb{M} \to \mathbb{M}$
as $\sym \tau = (\tau + \tau^\top)/2$, the skew operator $\skw: \mathbb{M} \to \mathbb{M}$ as $\skw \tau = (\tau - \tau^\top)/2$,
and the asymmetry operator $\asym: \mathbb{M} \to \mathbb{K}$ along with its adjoint $\asym^*:  \mathbb{K} \to \mathbb{M}$, as
\begin{align*}
    \asym \tau &= \begin{bmatrix}
        \tau_{32} - \tau_{23} \\
        \tau_{13} - \tau_{31} \\
        \tau_{21} - \tau_{12}
    \end{bmatrix} , &
    \asym^* r &=
    \begin{bmatrix}
        0    & -r_3 & r_2  \\
        r_3  & 0    & -r_1 \\
        -r_2 & r_1  & 0
    \end{bmatrix}  & \text{for } 
    d &= 3,
    \\
    \asym \tau                &= \tau_{21} - \tau_{12}, &
    \asym^* r                   &=
    \begin{bmatrix}
        0 & -r \\
        r & 0
    \end{bmatrix} &
    \text{for }
    d &= 2.
\end{align*}
We note the identities $\sym \tau + \skw \tau = \tau$, $\asym \asym^* r = 2r$ and $\asym^* \asym \tau = 2 \skw \tau$.
To finish this section, we introduce the invertible operator $S: \mathbb{W} \to \mathbb{W}$, given by $S \theta = \theta$ in 2D and $S \theta = \theta^T - (\theta : I) I$ in 3D. The following identity then holds for sufficiently regular $\theta \in \mathbb{W}$:
\begin{align} \label{eq: asym identity}
    \asym (\nabla \times \theta) = \div (S \theta),
\end{align}
in which $\nabla \times \theta = [\partial_2 \theta, -\partial_1 \theta]$ in 2D.

The notation $\alpha \lesssim \beta$ implies that a $c > 0$ exists, independent of the mesh size or the length scale $\ell$, such that $c \alpha \le \beta$. The relation ``$\gtrsim$'' has analogous meaning and $\alpha \eqsim \beta$ means that $\alpha \lesssim \beta \lesssim \alpha$.

\section{The Cosserat equations} \label{sec:model}

The primary variables are the Cauchy stress $\sigma \in \mathbb{M}$, the couple stress $\omega \in \mathbb{W}$, 
the displacement $u \in \mathbb{V}$, and the rotation $r \in \mathbb{K}$. As data for the problem, we introduce
the fourth-order material tensor $\mathcal{C}_\sigma: \mathbb{M} \to \mathbb{M}$ in a Cosserat material
\begin{align*}
    \mathcal{C}_\sigma \tau
    &\coloneqq 2 \mu_\sigma \sym \tau + 2\mu_\sigma^c \skw \tau + \lambda_\sigma (\Tr \tau) I,
    & \tau &\in \mathbb{M},
\end{align*}
in which $\mu_\sigma > 0$ and $\lambda_\sigma \ge 0$ are the Lam\'e parameters and $\mu_\sigma^c > 0$ is the Cosserat couple modulus. Moreover, for the couple stress $\omega$
we have the following material tensor $\mathcal{C}_\omega : \mathbb{W}\to \mathbb{W}$ as
\begin{align*}
    \mathcal{C}_\omega \tau &\coloneqq
    \begin{cases*}
        2 \mu_\omega \sym \tau + 2\mu_\omega^c \skw \tau + \lambda_\omega (\Tr \tau) I & $d = 3$,\\
        2 \mu_\omega \tau & $d=2$,
    \end{cases*}
    & \tau &\in \mathbb{W},
\end{align*}
with dedicated material parameters $\mu_\omega>0$, $\mu_\omega^c>0$,  and $\lambda_\omega \ge 0$ for $\omega$.
Let $\mathcal{A}_\sigma$ denote the inverse of $\mathcal{C}_\sigma$, given by
\begin{align} \label{eq: def A}
    \mathcal{A}_\sigma \tau 
    &\coloneqq \frac1{2\mu_\sigma} \left( \sym \tau - \frac{\lambda_\sigma}{2\mu_\sigma + d \lambda_\sigma} (\Tr \tau) I\right) + \frac1{2\mu_\sigma^c} \skw \tau.
\end{align}
Similarly, let $\mathcal{A}_\omega$ denote the inverse of $\mathcal{C}_\omega$. The tensors $\mathcal{A}_\sigma$ and $\mathcal{A}_\omega$ are positive definite and bounded, satisfying for all $\tau \in \mathbb{M}$, a.e. in $\Omega$,
\begin{equation}\label{A-pd}
  \mathcal{A}_\sigma\tau:\tau \eqsim \tau:\tau, \quad
  \mathcal{A}_\omega\tau:\tau \eqsim \tau:\tau.
  \end{equation}

Let $\ell \ge 0$ be a continuous parameter that represents the scale separation in the micropolar medium. We assume that
\begin{align} \label{eq: bound ell}
    \| \ell \|_{L^\infty(\Omega)} + \| \nabla \ell \|_{L^\infty(\Omega)} \lesssim 1.
\end{align}

\begin{remark}
It is assumed in \cite{boon2025mixed} that $\ell$ is piecewise linear. Here we avoid this assumption by taking a slightly modified approach in the analysis, see Remark~\ref{ell-remark} below.
\end{remark}

The strong formulation of the Cosserat problem is: find $(\sigma, \omega, u, r) \in \mathbb{M} \times \mathbb{W} \times \mathbb{V} \times \mathbb{K}$ such that
\begin{subequations} \label{eq:cosserat}
\begin{align}
    \mathcal{A}_\sigma \sigma - \nabla u - \asym^* r &= g_\sigma, &
    \mathcal{A}_\omega \omega - \ell \nabla r &= \ell g_\omega,
    & \text{in } &\Omega, \label{eq: constit laws} \\
    - \div \sigma &= f_\sigma, &
    \asym \sigma - \div \ell \omega &= f_\omega,
    & \text{in } &\Omega. \label{eq: momentum eqs}
\end{align}
with $g_\sigma$, $g_\omega$, $f_\sigma$, and $f_\omega$ given. 
We impose essential and natural boundary conditions as
\begin{align}
    \sigma n = 0
    \quad \text{and} \quad
    \ell \omega \cdot n =  0
    && \text{on } \partial_e \Omega,
    &&\qquad &&
    u = 0
    \quad \text{and} \quad
    r = 0
    && \text{on } \partial_n \Omega.
\end{align}
\end{subequations}
The boundary conditions are assumed to be zero for simplicity. Non-homogeneous essential boundary conditions can be handled by a lifting technique, while non-homogeneous natural boundary conditions result in additional boundary terms, which can be controlled for sufficiently smooth data. 
In the variational formulation of \eqref{eq:cosserat}, we seek the solution quadruplet $(\sigma, \omega, u, r) \in \Sigma \times W \times U \times R$ in the following Sobolev spaces
\begin{subequations} \label{eq: Sobolev spaces}
\begin{align} 
    \Sigma &\coloneqq \left\{\sigma \in H_{\divo} \mathbb{M}: n \cdot \sigma|_{\partial_e \Omega} = 0 \right\}, &
    W &\coloneqq \left\{\omega \in L^2 \mathbb{W}: \div \ell \omega \in L^2 \mathbb{K}, n \cdot \ell \omega|_{\partial_e \Omega} = 0 \right\}, \\
    U &\coloneqq L^2 \mathbb{V}, &
    R &\coloneqq L^2 \mathbb{K}.
\end{align}
\end{subequations}

The variational formulation of problem \eqref{eq:cosserat} is: find $(\sigma,\omega,u,r) \in \Sigma \times W \times U \times R$ such that
\begin{subequations}\label{eq:weak}
  \begin{align}
    (\mathcal{A}_\sigma \sigma, \sigma')_\Omega + (\div \sigma', u)_\Omega
    - (\asym \sigma', r)_\Omega & = (g_\sigma, \sigma')_\Omega & 
    \forall \sigma' &\in \Sigma, \label{eq:weak-1}\\
    (\mathcal{A}_\omega \omega, \omega')_\Omega + (\div \ell \omega', r)_\Omega & = (\ell g_\omega, \omega')_\Omega & 
    \forall \omega' &\in W,\label{eq:weak-2}\\
    -(\div \sigma, u')_\Omega & = (f_\sigma, u')_\Omega & 
    \forall u' &\in U,\label{eq:weak-3}\\
    (\asym \sigma, r')_\Omega - (\div \ell \omega, r')_\Omega & = (f_\omega, r')_\Omega & 
    \forall r' &\in R.\label{eq:weak-4}
  \end{align} 
\end{subequations}

For brevity, we collect the variables into two pairs and introduce the following notation for the product spaces
\begin{align}
	\eta &\coloneqq (\sigma, \omega) \in \Sigma \times W \eqqcolon X, &
	v &\coloneqq (u, r) \in U \times R \eqqcolon Y.
\end{align}
The product spaces $X$ and $Y$ are then endowed with the following $H_{\divo}$ and $L^2$-type norms:
\begin{align}\label{norms}
	\| \eta \|_X^2 &\coloneqq 
	\| \sigma \|_\Omega^2 + \| \div \sigma \|_\Omega^2
	+ \| \omega \|_\Omega^2 + \| \div \ell \omega \|_\Omega^2, &
	\| v \|_Y^2 &\coloneqq 
	\| u \|_\Omega^2 + \| r \|_\Omega^2.
\end{align}
The variational formulation \eqref{eq:weak} can now be concisely written as: find $(\eta, v) \in X \times Y$ such that
\begin{subequations}\label{eq:cosserat_weak}
\begin{align}
    \langle A\eta, \eta' \rangle
    - \langle B \eta', v \rangle
    & = \langle g, \eta' \rangle \quad \forall \eta' \in X,\\
    \langle B \eta, v' \rangle &= \langle f, v' \rangle \quad
    \forall  v' \in Y,
\end{align}
\end{subequations}
in which the operators $A: X \to X^\prime$ and $B: X \to Y^\prime$, and the functionals $g \in X^\prime$ and $f \in Y^\prime$ are given by
\begin{subequations}
\begin{align}
    \langle A \eta, \eta' \rangle
    &\coloneqq (\mathcal{A}_\sigma \sigma, \sigma')_\Omega + (\mathcal{A}_\omega \omega, \omega')_\Omega, \label{eq: operator A}\\
    \langle B \eta, v' \rangle
    &\coloneqq -(\div \sigma, u')_\Omega
    + (\asym \sigma, r')_\Omega
    - (\div \ell \omega, r')_\Omega, \\
    \langle g, \eta' \rangle &\coloneqq (g_\sigma, \sigma')_\Omega
    + (\bar u, \sigma')_{\partial_n \Omega}
    + (\ell g_\omega, \omega')_\Omega 
    + (\ell \bar r, \omega')_{\partial_n \Omega}
    , \\
    \langle f, v'\rangle  &\coloneqq (f_\sigma, u')_\Omega + (f_\omega, r')_\Omega, 
\end{align}
\end{subequations}
for all $\eta, \eta' \in X$ and $v, v' \in Y$. The well posedness of \eqref{eq:cosserat_weak} is established in \cite[Cor.~4.4]{boon2025mixed}. 

\begin{remark}
    The symmetries of $\mathcal{A}_\sigma$ and $\mathcal{A}_\omega$ are apparent when considered in a variational form. In particular, the identities at the end of \Cref{sub: preliminaries} allow us to write
    \begin{align*}
        (\mathcal{A}_\sigma \sigma, \sigma')_\Omega
        = \frac1{2\mu_\sigma} \left( ( \sigma, \sigma')_\Omega
        - \alpha_\sigma (\Tr \sigma, \Tr \sigma')_\Omega \right)
        + \beta_\sigma \frac12 (\asym \sigma, \asym \sigma')_\Omega
    \end{align*}
    with $\alpha_\sigma = {\lambda_\sigma}/ ({2\mu_\sigma + d \lambda_\sigma}) $ and $\beta_\sigma =  (\mu_\sigma - \mu_\sigma^c) / (2 \mu_\sigma \mu_\sigma^c)$. 
\end{remark}

In the limit case of $\ell = 0$, with zero $f_\omega$, the second equation of \eqref{eq: constit laws} becomes $\mathcal{A}_\omega \omega = 0$, which implies that $\omega = 0$. Moreover, the third equation implies $\asym \sigma = 0$ which relates to the conservation of angular momentum. Thus, \eqref{eq:cosserat} effectively degenerates to the linearized elasticity problem: find $(\sigma, u, r) \in \mathbb{M} \times \mathbb{V} \times \mathbb{K}$ such that
\begin{align}\label{eq:elasticity}
    \mathcal{A}_\sigma \sigma - \nabla u - \asym^* r &= g_\sigma, &
    - \div \sigma &= f_\sigma, &
    \asym \sigma &= 0,
    & \text{in } &\Omega.
\end{align}

\section{General analysis of multipoint stress mixed finite element methods}
\label{sec: MS-MFE}

In this section, we introduce the general strategy to construct stable and convergent multipoint stress mixed finite element methods for the Cosserat problem. These general results will be used to prove stability and convergence of the methods proposed in the subsequent four sections.

\subsection{Robust mixed finite element methods based on stable elasticity triplets}

Let $\Omega_h$ be a shape-regular, simplicial tessellation of $\Omega$, on which we consider the following finite element spaces. Let $\P{k}$ denote the element-wise, discontinuous polynomial finite elements on $\Omega_h$ of order $k$. Let $\L1 \subset \P1$ denote the lowest order Lagrange elements, containing continuous, piecewise linear functions.
Let $\RT{k}$ denote the Raviart-Thomas space of order $k$, for which $k$ denotes the polynomial order of the normal traces on the mesh facets. Similarly, let $\BDM{k}$ denote the Brezzi-Douglas-Marini space of order $k$.
For ease of reference, we recall the inclusions
\begin{align}
    \RT{k} &\subseteq \P{k + 1}^d, &
    \BDM{k} &\subseteq \P{k}^d, &
    \div \RT{k} = \div \BDM{k + 1} &\subseteq \P{k}.
\end{align}

Using these definitions, we will propose and analyze finite element spaces that are conforming in the sense that
$\Sigma_h \subseteq \Sigma$, 
$W_h \subseteq W$, 
$U_h \subseteq U$, and
$R_h \subseteq R$. 
Details further specifying each variant will be presented in \Cref{sec:simple_scheme,sec:reducible_scheme,sec:RT1-L1}.

For each choice of finite element spaces, we define $X_h \coloneqq \Sigma_h \times W_h$ and $Y_h \coloneqq U_h \times R_h$, in analogy with the continuous case. The mixed finite element problem is then posed as: find $(\eta_h, v_h) \in X_h \times Y_h$ such that
\begin{align}\label{eq: system full-MFE}
    \langle A\eta_h, \eta_h' \rangle
    - \langle B \eta_h', v_h \rangle
    + \langle B \eta_h, v_h' \rangle
    &= \langle g, \eta_h' \rangle + \langle f, v_h' \rangle, &
    \forall (\eta_h', v_h') &\in X_h \times Y_h.
    \end{align}

We will base our choice of discrete spaces on finite elements that form a stable discretization for elasticity with weakly imposed symmetry \cite{arnold2007mixed}. In particular, we choose spaces that satisfy the following property.

\begin{definition} \label{def:elasticity triplet}
    A triplet $\Sigma_h \times U_h \times R_h$ is \emph{elasticity-stable} if the following condition holds
    \begin{align}\label{eq:inf-sup}
        \inf_{(u_h, r_h) \in U_h \times R_h} \sup_{\sigma_h \in \Sigma_h}
        \frac{
            (\div \sigma_h, u_h)_\Omega + (\asym \sigma_h, r_h)_\Omega
        }{\| (\sigma_h, 0) \|_X \| (u_h, r_h) \|_Y}
        \gtrsim 1.
    \end{align}
\end{definition}

As noted in \cite[Thm.~4.6]{boon2025mixed}, elasticity-stable finite element spaces can directly be used to form a stable discretization for the Cosserat equations, regardless of the choice of $W_h$. Using this observation, we can formulate sufficient conditions to guarantee stability of \eqref{eq: system full-MFE}. For the stability analysis, we introduce the following discrete norm
\begin{align} \label{eq: discrete norm}
    \| \eta \|_{X_h}^2 \coloneqq 
    \| \sigma \|_\Omega^2
    + \| \div \sigma \|_\Omega^2
    + \| \omega \|_\Omega^2
    + \| \Pi_R \div \ell \omega \|_\Omega^2,
\end{align}
where $\Pi_R$ denotes the $L^2$ projection onto $R_h$. The associated dual norm is defined as $\| f \|_{X_h'} \coloneqq \sup_{\eta \in X_h} \frac{\langle f, \eta \rangle}{\| \eta \|_{X_h}}$ for $f \in X_h'$.

\begin{remark}
  In general, weakening of the norm of the divergence of the couple stress may result in violating the continuity of the operator $B$. In our case, as shown in \Cref{thm: stability full MFE} below, the use of $\| \Pi_R \div \ell \omega \|_\Omega$ in the norm $\| \eta \|_{X_h}$ still results in continuous operator $B$. Furthermore, due to the weakened norm, the stability result in \Cref{thm: stability full MFE} does not require that $\div W_h \subseteq R_h$. However, the accuracy of the method may be affected if this property does not hold, as we will see in the methods presented in \Cref{sec:reducible_scheme,sec:RT1-L1}.
\end{remark}

\begin{theorem}[Stability] \label{thm: stability full MFE}
    Let the pair $\Sigma_h \times U_h$ satisfy $\div \Sigma_h \subseteq U_h$ and let the triplet $\Sigma_h \times U_h \times R_h$ be elasticity-stable, cf.~\Cref{def:elasticity triplet}. Then Problem \eqref{eq: system full-MFE} admits a unique solution that satisfies
    \begin{align}
        \| \eta_h \|_{X_h} + \| v_h \|_Y \lesssim \| g \|_{X_h'} + \| f \|_{Y'}.
    \end{align}
\end{theorem}
\begin{proof}
    We verify the Brezzi conditions for saddle point problems \cite[Sec.~4.2.3]{boffi2013mixed}. First, we verify that $A$ and $B$ are continuous in the relevant norms
    \begin{align*}
        \langle A \eta_h, \eta_h' \rangle
        &\lesssim \| \eta_h \|_\Omega \| \eta_h' \|_\Omega
        \le \| \eta_h \|_{X_h} \| \eta_h' \|_{X_h}, &
        \forall \eta_h, \eta_h' &\in X_h, \\
        \langle B \eta_h, r_h \rangle
        &\le \| \div \sigma_h \|_\Omega \| u_h \|_\Omega
        + \| \asym \sigma_h \|_\Omega \| r_h \|_\Omega
        + \| \Pi_R \div \ell \omega_h \|_\Omega \| r_h \|_\Omega 
        \le \| \eta_h \|_{X_h} \| v_h \|_Y, &
        \forall (\eta_h, v_h) &\in X_h \times Y_h.
    \end{align*}

    Next, we verify the coercivity of $A$ on the kernel of $B$. Let $\eta_h = (\sigma_h, \omega_h)$ satisfy $\langle B \eta_h, v_h \rangle = 0$ for all $v_h \in Y_h$. Then $\div \Sigma_h \subseteq U_h$ implies
    \begin{align} \label{eq: Ker B}
        \div \sigma_h &= 0, &
        \Pi_R \div \ell \omega_h &= \Pi_R \asym \sigma_h,
    \end{align}
    which implies that
    \begin{align}
        \| \div \sigma_h \|_\Omega^2 + 
        \| \Pi_R \div \ell \omega_h \|_\Omega^2
        = \| \Pi_R \asym \sigma_h \|_\Omega^2
        \lesssim \| \sigma_h \|_\Omega^2.
    \end{align}
    Thus, for $\eta_h$ in the kernel of $B$, we have
    \begin{align}
        \langle A \eta_h, \eta_h \rangle
        \eqsim \| \sigma_h \|_\Omega^2 + \| \omega_h \|_\Omega^2
        \gtrsim 
        \| \sigma_h \|_\Omega^2 
        + \| \div \sigma_h \|_\Omega^2  
        + \| \omega_h \|_\Omega^2
        +\| \Pi_R \div \ell \omega_h \|_\Omega^2
        = \| \eta_h \|_{X_h}^2.
    \end{align}

    We continue by showing that $B$ satisfies an inf-sup condition on $X_h \times Y_h$ in the relevant norms. As noted in \cite[Lem.~4.5]{boon2025mixed}, this follows immediately if $\Sigma_h \times U_h \times R_h$ is elasticity-stable. Using the fact that $\| (\sigma_h, 0) \|_{X_h} = \| (\sigma_h, 0) \|_X$, we derive
    \begin{align} \label{eq: infsup from elasticity}
        \inf_{v_h \in Y_h} \sup_{\eta_h \in X_h}
        \frac{
            \langle B \eta_h, v_h \rangle
        }{\| \eta_h \|_{X_h} \| v_h \|_Y}
        \ge
        \inf_{v_h \in Y_h} \sup_{(\sigma_h, 0) \in X_h}
        \frac{
            \langle B (\sigma_h, 0), v_h \rangle
        }{\| (\sigma_h, 0) \|_{X_h} \| v_h \|_Y}
        =
        \inf_{(u_h, r_h) \in U_h \times R_h} \sup_{\sigma_h \in \Sigma_h}
        \frac{
            (\div \sigma_h, u_h)_\Omega + (\asym \sigma_h, r_h)_\Omega
        }{\| (\sigma_h, 0) \|_X \| (u_h, r_h) \|_Y}
        \gtrsim 1.
    \end{align}
\end{proof}

The stability constant in \Cref{thm: stability full MFE} is independent of $\ell \ge 0$. This robustness of the mixed finite element method with respect to the length scale $\ell$ is captured in the following definition.

\begin{definition} \label{def: l-robust}
    A discretization method for the linear Cosserat system is \emph{$\ell$-robust} if it is stable in the limit case $\ell = 0$. 
\end{definition}

To prove convergence of the mixed finite element method, we require aspects of the finite element spaces that are not available at this stage. We therefore postpone those results to \Cref{sec:simple_scheme,sec:reducible_scheme,sec:RT1-L1,sec:RT1-P1}, where these spaces are specified.

\subsection{Multipoint stress mixed finite element methods based on low-order integration}

To formulate the multipoint stress mixed finite element method, we introduce discrete inner products that employ low-order quadrature rules. We will use two discrete inner products in particular, defined as follows.

\begin{definition} \label{def:quadrature}
    For an element $\Delta \in \Omega_h$, let $|\Delta|$ be its measure, $\mathcal{N}(\Delta)$ its node set, and $x_\Delta$ the element center.
    For $\phi, \phi' \in \P2$, we introduce the following inner products and induced norms
    \begin{align}
    	(\phi , \phi' )_{Q_1} &\coloneqq
        \sum_{\Delta \in \Omega_h} \frac{|\Delta|}{d + 1}
        \sum_{x_i \in \mathcal{N}(\Delta)} \phi_\Delta(x_i) \cdot \phi'_\Delta(x_i),  &
        \| \phi \|_{Q_1} &\coloneqq \sqrt{(\phi, \phi)_{Q_1}}, \label{eq:Q1}\\
        (\phi , \phi' )_{Q_2} &\coloneqq
        \frac{1}{d + 2}\left((\phi , \phi' )_{Q_1} +
        (d + 1) \sum_{\Delta \in \Omega_h} 
             |\Delta| \phi(x_\Delta) \cdot \phi'(x_\Delta)            
        \right), &
        \| \phi \|_{Q_2} &\coloneqq \sqrt{(\phi, \phi)_{Q_2}}. \label{eq:Q2} 
    \end{align}
    Here, $\phi_\Delta \coloneqq \phi|_\Delta$ denotes a restriction to the element.
\end{definition}

These discrete inner products have several important properties, which we summarize in the following lemmas.

\begin{lemma}[{\cite[Thm.~4.1]{lee2018local}}] \label{lem: quadrature_1}
	On $\P1$, the norm $\| \cdot \|_{Q_1}$ is equivalent to the $L^2(\Omega)$-norm and the integration rule from \eqref{eq:Q1} is exact for piecewise linear functions. In other words
	\begin{align}
		\| \phi_1 \|_{Q_1} &\eqsim \| \phi_1 \|_\Omega, &
		(\phi_1, \varphi_0)_{Q_1} &= 
		(\phi_1, \varphi_0)_\Omega, &
		\forall \phi_1 &\in \P1, \forall \varphi_0 \in \P0.
	\end{align}
\end{lemma}

\begin{lemma}[{\cite{egger2020second}}] \label{lem: quadrature_2}
    On $\P2$, the norm $\| \cdot \|_{Q_2}$ is equivalent to the $L^2(\Omega)$-norm and the integration rule from \eqref{eq:Q2} is exact for piecewise quadratic functions. In other words
    \begin{align}
        \| \phi_2 \|_{Q_2} &\eqsim \| \phi_2 \|_\Omega, &
        (\phi_2, \varphi_0)_{Q_1} &= 
        (\phi_2, \varphi_0)_\Omega, &
        \forall \phi_2 &\in \P2, \forall \varphi_0 \in \P0.
    \end{align}
\end{lemma}

We will consider discrete spaces $\Sigma_h \subseteq \P2^{d \times d}$, $W_h \subseteq \P2^{d \times k_d}$, $U_h \subseteq \P1^d$, and $R_h \subseteq \P1^{k_d}$. By slightly abusing notation, we extend the discrete inner products of \Cref{def:quadrature} to these tensor- and vector-valued discrete spaces.
The quadrature rule allows us to construct $A_h: X_h \to X_h'$ as an approximation of the operator $A$ from \eqref{eq: operator A}:
\begin{align} \label{eq: A_h}
    \langle A_h \eta_h , \eta_h'  \rangle
    &\coloneqq (\mathcal{A}_\sigma \sigma_h, \sigma_h')_Q + (\mathcal{A}_\omega \omega_h, \omega_h')_Q.
\end{align}
The multipoint stress mixed finite element (MS-MFE) method considered in three of the four cases we study is: find $(\hat \eta_h, \hat v_h) \in X_h \times Y_h$ such that
\begin{align} \label{eq: system MS-MFE}
    \langle A_h \hat \eta_h, \eta_h' \rangle
    - \langle B \eta_h', \hat v_h \rangle
    + \langle B \hat \eta_h, v_h' \rangle
    &= 
    \langle g, \eta_h' \rangle
    + \langle f, v_h' \rangle, &
    \forall (\eta_h', v_h') &\in X_h \times Y_h.
\end{align}
In one of the four cases we study, cf. \Cref{sec:reducible_scheme}, a quadrature rule will be applied also to the rotation bilinear forms, which results in a modified operator $B_h$. The stability and convergence theorems presented next will not be applicable for this method and specific analysis will be developed in \Cref{sec:reducible_scheme}.

\begin{theorem}[Stability] \label{thm: stability MS-MFE}
    If the assumptions of \Cref{thm: stability full MFE} are met, $A_h$ is defined by \eqref{eq: A_h}, and 
    \begin{align} \label{eq: norm equivalence Q}  
        \| \eta_h \|_Q &\eqsim \| \eta_h \|_\Omega, & 
        \forall \eta_h &\in \Sigma_h,
    \end{align}
    then the MS-MFE method is stable, i.e. Problem \eqref{eq: system MS-MFE} admits a unique solution that satisfies the bound:
    \begin{align}\label{stab-bound}
        \| \hat \eta_h \|_{X_h} + \| \hat v_h \|_Y \lesssim \| g \|_{X_h'} + \| f \|_{Y'}.
    \end{align}
\end{theorem}
\begin{proof}
Since the operator $B$ remains unchanged with respect to the MFE Problem \eqref{eq: system full-MFE}, we only need to consider the continuity and coercivity of $A_h$. Both follow by the arguments from \Cref{thm: stability full MFE}, combined with the norm equivalence \eqref{eq: norm equivalence Q}. For clarity, we demonstrate its continuity:
    \begin{align*}
        \langle A_h \eta_h, \eta_h' \rangle
        &\lesssim \| \eta_h \|_Q \| \eta_h' \|_Q
        \eqsim \| \eta_h \|_\Omega \| \eta_h' \|_\Omega
        \le \| \eta_h \|_{X_h} \| \eta_h' \|_{X_h}, &
        \forall \eta_h, \eta_h' &\in X_h,
    \end{align*}
    The stability result is now obtained by invoking saddle point theory \cite[Sec.~4.2.3]{boffi2013mixed}.
\end{proof}

\begin{theorem}[Convergence] \label{thm:conv MS-MFE}
    Let the conditions of \Cref{thm: stability MS-MFE} be met and let the quadrature rule be such that
    \begin{align} \label{eq: exactness quadrature Q}
        (\eta_h, \phi_0)_Q &= (\eta_h, \phi_0)_\Omega, &
        \forall \eta_h &\in X_h, \phi_0 \in \P0^{d \times d} \times \P0^{d \times k_d}.
    \end{align}
    Let the solution $(\eta, v)$ to \eqref{eq:cosserat_weak} be sufficiently regular and let the mixed finite element method \eqref{eq: system full-MFE} satisfy the linear convergence estimate
    \begin{align} \label{eq: assumed conv MFE}
        \| \eta_h - \eta \|_\Omega + \| v_h - v \|_Y \lesssim h.
    \end{align}
    Then the MS-MFE method \eqref{eq: system MS-MFE} converges linearly as well, i.e.
    \begin{align}
        \| \hat \eta_h - \eta \|_{X_h} + \| \hat v_h - v \|_Y \lesssim h.
    \end{align}
\end{theorem}
\begin{proof}
    Following \cite[Thm.~3.2]{lee2018local}, we subtract \eqref{eq: system full-MFE} from \eqref{eq: system MS-MFE} and add the term $\langle (A - A_h) \eta_h, \eta_h' \rangle$ to both sides to obtain
    \begin{align} \label{eq: error equations}
        \langle A_h (\hat \eta_h - \eta_h), \eta_h' \rangle
        - \langle B \eta_h', \hat v_h - v_h \rangle
        + \langle B (\hat \eta_h - \eta_h), v_h' \rangle
        &= 
        \langle (A - A_h) \eta_h, \eta_h' \rangle, &
        \forall (\eta_h', v_h') &\in X_h \times Y_h.
    \end{align}

    Thus the pair $(\hat \eta_h - \eta_h, \hat v_h - v_h) \in X_h \times Y_h$ is the solution to \eqref{eq: system MS-MFE} with right-hand side $g = (A - A_h) \eta_h$ and $f = 0$. This allows us to apply the stability estimate \eqref{stab-bound}:
    \begin{align} \label{eq: diff discrete sols}
        \| \hat \eta_h - \eta_h \|_{X_h} + \| \hat v_h - v_h \|_Y
        &\lesssim \sup_{\eta_h' \in X_h} \frac{\langle (A - A_h) \eta_h, \eta_h' \rangle}{\| \eta_h' \|_{X_h}}.
    \end{align}
    
    Next, we introduce $\Pi_0$ as the $L^2$ projection onto the piecewise constants $\P0^{d \times d} \times \P0^{d \times k_d}$. The identity \eqref{eq: exactness quadrature Q} implies $\langle (A - A_h) \Pi_0 \eta_h, \eta_h' \rangle = 0$. Using this in combination with $\| \eta'_h \|_\Omega \le \| \eta_h' \|_{X_h}$, we derive
    \begin{align} \label{eq: bound diff A}
        \sup_{\eta_h' \in X_h} \frac{\langle (A - A_h) \eta_h, \eta_h' \rangle}{\| \eta_h' \|_{X_h}}
        &\le \sup_{\eta_h' \in X_h} \frac{\langle (A - A_h) \eta_h, \eta_h' \rangle}{\| \eta_h' \|_\Omega}
        = \sup_{\eta_h' \in X_h} \frac{\langle (A - A_h) (\eta_h - \Pi_0 \eta), \eta_h' \rangle}{\| \eta_h' \|_\Omega} 
        \nonumber \\
        &\lesssim \| \eta_h - \Pi_0 \eta \|_\Omega 
        \le \| \eta_h - \eta \|_\Omega + \| (I - \Pi_0) \eta \|_\Omega
        \lesssim \| \eta_h - \eta \|_\Omega + h \| \eta \|_{1, \Omega}.
    \end{align}
    In the final steps, we used the continuity of $A$ and $A_h$ in $L^2(\Omega)$ and the approximation properties of $\P0$ in $H^1(\Omega)$. 
    To finish the proof, we use a triangle inequality with \eqref{eq: diff discrete sols}, \eqref{eq: bound diff A}, and the assumed linear convergence \eqref{eq: assumed conv MFE}.
    \begin{align}
        \| \hat \eta_h - \eta \|_{X_h} + \| \hat v_h - v \|_Y 
        &\le \| \hat \eta_h - \eta_h \|_{X_h} + \| \hat v_h - v_h \|_Y + \| \eta_h - \eta \|_{X_h} + \| v_h - v \|_Y \nonumber \\
        &\lesssim \| \eta_h - \eta \|_\Omega + \| v_h - v \|_Y + h \| \eta \|_{1, \Omega} \lesssim h.
    \end{align}
\end{proof}

For appropriate choices of finite element spaces, the matrix associated with $A_h$ becomes block-diagonal, and thereby easily invertible. This allows us to consider the equivalent, Schur-complement system: find $\hat v_h \in Y_h$ such that
\begin{subequations}
\label{eq: Schur-complement}
\begin{align} 
    \langle B A_h^{-1} B^* \hat v_h, v_h' \rangle
    &= 
    \langle f - B A_h^{-1} g, v_h' \rangle, &
    \forall v_h' &\in Y_h.
\end{align}
The stresses $\eta_h$ can then be post-processed by solving the block-diagonal system 
\begin{align}
    \langle A_h \hat \eta_h, \eta_h' \rangle &= \langle B \eta_h', \hat v_h \rangle + \langle g, \eta_h' \rangle, &
    \forall \eta_h' &\in X_h.
\end{align}
\end{subequations}

In the following sections, we will present and analyze four $\ell$-robust schemes: 
a \emph{simple} scheme to which the theory of this section directly applies, 
a \emph{reducible} scheme based on \cite{ambartsumyan2020multipoint} that allows for elimination of the rotation variable, 
and two \emph{higher-order} schemes based on the quadrature rule from \cite{egger2020second}, with either continuous or discontinuous rotations.
For ease of reference, we summarize the methods under consideration and our theoretical convergence estimates in \Cref{tab: summary}. 

\begin{table}[ht]
    \caption{Summary of the multipoint stress mixed finite element schemes considered in this work and their convergence orders. The method names refer to the finite elements used to discretize the stress and rotation spaces. The notation 1+$1_{\ell = 0}$ denotes second order convergence for the elasticity system with $\ell = 0$.}
    \label{tab: summary}
    \centering

    \begin{tabular}{cc|cccc|cccc}
    \hline

    \hline
     Name & Section & $\Sigma_h$ & $W_h$ & $U_h$ & $R_h$ & $\hat \sigma_h$ & $\hat \omega_h$ & $\hat u_h$ & $\hat r_h$ \\
    \hline
        $\BDM1$-$\P0$ & \ref{sec:simple_scheme} &
        $\BDM1$ & $\BDM1$ & $\P0$ & $\P0$ & 1 & 1 & 1 & 1\\
        $\BDM1$-$\L1$ & \ref{sec:reducible_scheme} &
        $\BDM1$ & $\BDM1$ & $\P0$ & $\L1$ & 1 & 1 & 1 & 1\\
        $\RT1$-$\L1$ & \ref{sec:RT1-L1} &
        $\RT1$ & $\RT1$ & $\P1$ & $\L1$ & 1+$1_{\ell = 0}$ & 1 & 1 & 1+$1_{\ell = 0}$\\
        $\RT1$-$\P1$ & \ref{sec:RT1-P1} &
        $\RT1$ & $\RT1$ & $\P1$ & $\P1$ & 2 & 1 & 1 & 2 \\
    \hline
    \end{tabular}
\end{table}

\begin{remark}\label{rem:conservation BDM1P0}
In three of the cases, where the quadrature rule is used only in the stress bilinear forms, the multipoint stress methods do not introduce an additional error in the
momentum balance equations \eqref{eq: momentum eqs}. In particular, the discrete solutions $\eta_h$ and $\hat \eta_h$ satisfy $\langle B \eta_h, v_h' \rangle = \langle f, v_h' \rangle = \langle B \hat \eta_h, v_h' \rangle$ for all $v_h' \in V_h$. Moreover, since the displacements are discontinuous, linear momentum balance is enforced locally on each element. The same holds for the balance of angular momentum for the methods with discontinuous rotations. In the method of \Cref{sec:reducible_scheme}, the angular momentum is enforced through the vertex quadrature rule.
\end{remark}

\section{A simple scheme: \texorpdfstring{$\BDM1$-$\P0$}{BDM1-P0}}
\label{sec:simple_scheme}

The first method is based on one of the choices of finite element spaces considered in \cite{boon2025mixed}:
\begin{align} \label{eq: Spaces BDM1-P0}
    \Sigma_h &\coloneqq \BDM1^d \cap \Sigma, &
    W_h &\coloneqq \BDM1^{k_d} \cap W, &
    U_h &\coloneqq \P0^d, &
    R_h &\coloneqq \P0^{k_d}.
\end{align}
The intersections in the definitions of the first two spaces ensure that the essential boundary conditions are respected, cf.~\eqref{eq: Sobolev spaces}. We will refer to the methods of this section as ``$\BDM1$-$\P0$'' after the finite element spaces used to discretize $\Sigma$ and $R$. We note that the above choice of finite element spaces satisfies
\begin{equation}\label{eq:div-prop}
\div W_h \subseteq R_h,
\end{equation}
which will be utilized in the error analysis.

\subsection{The mixed finite element method}

\begin{remark}\label{ell-remark}
    Recognizing that the Cosserat equations are a Hodge-Laplace problem \cite[Sec.~5.2]{vcap2024bgg}, the well-posedness of \eqref{eq: system full-MFE} was proven in \cite[Thm.~4.6]{boon2025mixed} for the spaces in \eqref{eq: Spaces BDM1-P0}, under
    the assumption that $\ell$ is piecewise linear, using saddle point theory \cite[Sec. 4.2.3]{boffi2013mixed}. 
    We avoid this assumption by taking a slightly modified approach, using $\|\cdot\|_{X_h}$ of \eqref{eq: discrete norm} instead of the full norm $\|\cdot\|_X$ of \eqref{norms}.
\end{remark}

\begin{theorem}[Stability] \label{thm: stability full BDM1-P0}
    With the finite element spaces from \eqref{eq: Spaces BDM1-P0}, Problem \eqref{eq: system full-MFE} admits a unique solution that satisfies
    \begin{align}
        \| \eta_h \|_{X_h} + \| v_h \|_Y \lesssim \| g \|_{X_h'} + \| f \|_{Y'}.
    \end{align}
\end{theorem}
\begin{proof}
    The inclusion $\div \Sigma_h \subseteq U_h$ holds and the triplet $\Sigma_h \times U_h \times R_h$ was shown to be elasticity-stable in \cite{arnold2007mixed}. Therefore \Cref{thm: stability full MFE} applies.
\end{proof}

We continue with the convergence analysis of the mixed finite element method. While we employ standard arguments, the use of the discrete norm \eqref{eq: discrete norm} requires us to take additional care.

\begin{theorem}[Convergence] \label{thm:conv full BDM1-P0}
    If the solution $(\eta, v)$ to \eqref{eq:cosserat_weak} is sufficiently regular and the finite element spaces are chosen as in \eqref{eq: Spaces BDM1-P0}, then the mixed finite element solution $(\eta_h, v_h)$ of \eqref{eq: system full-MFE} satisfies
    \begin{align}
        \| \eta_h - \eta \|_{X_h} + \| v_h - v \|_Y \lesssim h.
    \end{align}
\end{theorem}
\begin{proof}
    From \eqref{eq:cosserat} and \eqref{eq: system full-MFE}, we deduce
    \begin{align}
        \langle A \eta_h, \eta_h' \rangle
        - \langle B \eta_h', v_h \rangle
        + \langle B \eta_h, v_h' \rangle
        &= 
        \langle A \eta, \eta_h' \rangle
        - \langle B \eta_h', v \rangle
        + \langle B \eta, v_h' \rangle.
    \end{align}
    Let $\pi_X = (\pi_\Sigma, \pi_W)$ denote the canonical interpolant onto $X_h$, inherent to $\BDM1$, and let $\Pi_Y = (\Pi_U, \Pi_R)$ be the $L^2$ projection onto $Y_h$. Subtracting these from both sides gives us
    \begin{align} \label{eq: three terms}
        \langle A (\eta_h - \pi_X \eta), \eta_h' \rangle
        - \langle B \eta_h', v_h - \Pi_Y v \rangle
        + \langle B \eta_h - \pi_X \eta, v_h' \rangle
        &= 
        \langle A (I - \pi_X) \eta, \eta_h' \rangle
        - \langle B \eta_h', (I - \Pi_Y) v \rangle
        + \langle B (I - \pi_X) \eta, v_h' \rangle,
    \end{align}
    for all $(\eta_h', v_h') \in X_h \times Y_h$. This means that $(\eta_h - \pi_X \eta, v_h - \Pi_Y v)$ solves Problem \eqref{eq: system full-MFE} with right-hand sides $f = B (I - \pi_X) \eta$ and $g = A (I - \pi_X) \eta - B^* (I - \Pi_Y) v$. The stability estimate from \Cref{thm: stability full BDM1-L1} therefore applies and we continue by bounding the dual norms of the right-hand sides. The first term is bounded by the continuity of $A$ in $L^2$ and the approximation properties of $\pi_X$:
    \begin{align} \label{eq: bound term 1 conv}
        \langle A (I - \pi_X) \eta, \eta_h' \rangle 
        \lesssim \| (I - \pi_X) \eta \|_\Omega \| \eta_h' \|_\Omega
        \lesssim h \| \eta_h' \|_{X_h}.
    \end{align}
    
    For the second term, we substitute the definitions and use $\div \Sigma_h = U_h$ to derive
    \begin{align}
        \langle B \eta_h', (I - \Pi_Y) v \rangle
        &= -(\div \sigma_h', (I - \Pi_U) u)_\Omega
        + (\asym \sigma_h', (I - \Pi_R) r)_\Omega
        - (\div \ell \omega_h', (I - \Pi_R) r)_\Omega \nonumber \\
        &= (\asym \sigma_h', (I - \Pi_R) r)_\Omega
        - ((\nabla \ell) \cdot \omega_h', (I - \Pi_R) r)_\Omega
        - (\ell \div \omega_h', (I - \Pi_R) r)_\Omega.  \label{eq: bound term 2}
    \end{align}
    We continue in two steps. First, we use the Cauchy-Schwarz inequality, the bound on $\ell$ from \eqref{eq: bound ell}, and the approximation properties of $R_h$ to derive
    \begin{align}
        (\asym \sigma_h', (I - \Pi_R) r)_\Omega
        - ((\nabla \ell) \cdot \omega_h', (I - \Pi_R) r)_\Omega
        &\le (\| \sigma_h' \|_\Omega + \| (\nabla \ell) \cdot \omega_h' \|_\Omega ) \| (I - \Pi_R) r \|_\Omega \nonumber \\
        &\lesssim (\| \sigma_h' \|_\Omega + \| \omega_h' \|_\Omega ) h. \label{eq: bound term 2a}
\end{align}
For the final term of \eqref{eq: bound term 2}, using \eqref{eq:div-prop}, we introduce $\Pi_0$ as the $L^2$ projection onto $\P0$ and use a discrete inverse inequality to obtain
    \begin{align}
        (\ell \div \omega_h', (I - \Pi_R) r)_\Omega
        &= ((\ell - \Pi_0 \ell) \div \omega_h', (I - \Pi_R) r)_\Omega \nonumber \\
        &\lesssim \| (I - \Pi_0) \ell \|_{L^\infty(\Omega)} \| \div \omega_h' \|_\Omega \| (I - \Pi_R) r \|_\Omega \nonumber\\
        &\lesssim h h^{-1} \| \omega_h' \|_\Omega h
        = h \| \omega_h' \|_\Omega. \label{eq: bound problematic term}
    \end{align}
    Note that we require the inverse inequality because the $X_h$-norm only controls $\| \Pi_R \div \ell \omega_h' \|$, and not $\| \div \omega_h' \|$. 
    Together, \eqref{eq: bound term 2a} and \eqref{eq: bound problematic term} allow us to bound \eqref{eq: bound term 2} as
    \begin{align} \label{eq: bound term 2 conv}
        \langle B \eta_h', (I - \Pi_Y) v \rangle \lesssim h \| \eta_h' \|_{X_h}.
    \end{align}

    For the third term in \eqref{eq: three terms}, we use the commuting property $\Pi_U \div (I - \pi_\Sigma) = 0$, and the bounds on $\ell$:
    \begin{align} \label{eq: bound term 3 conv}
        \langle B (I - \pi_X) \eta, v_h' \rangle
        &= -(\div (I - \pi_\Sigma) \sigma, u_h')_\Omega
        + (\asym (I - \pi_\Sigma) \sigma, r_h')_\Omega
        - (\div \ell (I - \pi_W) \omega, r_h')_\Omega \nonumber \\
        &\le \| (I - \pi_\Sigma) \sigma \|_\Omega \| r_h' \|_\Omega
        + \| \div \ell (I - \pi_W) \omega \|_\Omega \| r_h' \|_\Omega \nonumber\\
        &\lesssim \left(
            \| (I - \pi_\Sigma) \sigma \|_\Omega
            + \| (I - \pi_W) \omega \|_\Omega) + \| \div (I - \pi_W) \omega \|_\Omega\right) 
            \| r_h' \|_\Omega 
        \lesssim h \| v_h' \|_Y.
    \end{align}
    In the final step, we used the approximation properties of $\pi_X$. 
    We now collect the bounds on the three terms to deduce
    \begin{align}
        \| \eta_h - \pi_X \eta \|_{X_h} + \| v_h - \Pi_Y v \|_Y \lesssim 
        \sup_{\eta_h' \in X_h} \frac{\langle A (I - \pi_X) \eta, \eta_h' \rangle - \langle B \eta_h', (I - \Pi_Y) v \rangle}{\| \eta_h' \|_{X_h}}
        + \sup_{v_h' \in Y_h} \frac{\langle B (I - \pi_X) \eta, v_h' \rangle}{\| v_h' \|_{X_h}}
        \lesssim h.
    \end{align}

    It remains to bound the interpolation error $\| (I - \pi_X) \eta \|_{X_h} + \| (I - \Pi_Y) v \|_Y$ by $h$. This follows immediately from the approximation properties of $\pi_X$ and $\Pi_Y$ for the finite element spaces \eqref{eq: Spaces BDM1-P0}. The only exception is the term concerning $\| \Pi_R \div \ell \omega \|_\Omega$, which we bound as follows:
    \begin{align}
        \| \Pi_R \div \ell (I - \pi_W) \omega \|_\Omega
        \lesssim 
        \| (I - \pi_W) \omega \|_\Omega
        + \| \div (I - \pi_W) \omega \|_\Omega
        \lesssim h.
    \end{align}
    That concludes the proof.
\end{proof}

\subsection{The multipoint stress mixed finite element method}

To construct our multipoint stress mixed finite element method, we use the quadrature rule $Q_1$ from \eqref{eq:Q1} to define a discrete approximation of the operator $A$ (cf. \eqref{eq: operator A}):
\begin{align} \label{eq: A_h BDM1}
    \langle A_h \eta_h , \eta_h'  \rangle
    &\coloneqq (\mathcal{A}_\sigma \sigma_h, \sigma_h')_{Q_1} + (\mathcal{A}_\omega \omega_h, \omega_h')_{Q_1}.
\end{align}
The $\BDM1$-$\P0$ MS-MFE method now solves \eqref{eq: system MS-MFE}, which we repeat here for convenience: find $(\hat \eta_h, \hat v_h) \in X_h \times Y_h$ such that
\begin{align} \label{eq: system BDM P0}
    \langle A_h \hat \eta_h, \eta_h' \rangle
    - \langle B \eta_h', \hat v_h \rangle
    + \langle B \hat \eta_h, v_h' \rangle
    &= 
    \langle g, \eta_h' \rangle
    + \langle f, v_h' \rangle, &
    \forall (\eta_h', v_h') &\in X_h \times Y_h.
\end{align}

\begin{theorem}[Stability] \label{thm: stab MS BDM1-P0}
    The $\BDM1$-$\P0$ MS-MFE method is stable, i.e. Problem \eqref{eq: system BDM P0} admits a unique solution that satisfies the bound
    \begin{align}
        \| \hat \eta_h \|_{X_h} + \| \hat v_h \|_Y \lesssim \| g \|_{X_h'} + \| f \|_{Y'}.
    \end{align}
\end{theorem}
\begin{proof}
    Due to the inclusion $\BDM1 \subseteq \P1^d$, \Cref{lem: quadrature_1} gives us the norm equivalence \eqref{eq: norm equivalence Q}. The result therefore follows by \Cref{thm: stability MS-MFE}.
\end{proof}

We remark that in the case of $\ell = 0$, the scheme reduces to the MSMFE-0 method of \cite[Sec.~3]{ambartsumyan2020multipoint}. The stability of that method confirms that the $\BDM1$-$\P0$ MS-MFE method is $\ell$-robust in the sense of \Cref{def: l-robust}.

\begin{remark}\label{rem:norm}
Under the assumption that $\ell$ is piecewise linear, the argument from \cite[Thm.~4.6]{boon2025mixed} can be used, which gives the full-norm stability bound
$$
\| \hat \eta_h \|_{X} + \| \hat v_h \|_Y \lesssim \| g \|_{X'} + \| f \|_{Y'}.
$$
\end{remark}

\begin{theorem}[Convergence] \label{thm:conv MS BDM1-P0}
    If the solution $(\eta, v)$ to \eqref{eq:cosserat_weak} is sufficiently regular, then the $\BDM1$-$\P0$ MS-MFE method \eqref{eq: system BDM P0} converges linearly, i.e.
    \begin{align}
        \| \hat \eta_h - \eta \|_{X_h} + \| \hat v_h - v \|_Y \lesssim h.
    \end{align}
\end{theorem}
\begin{proof}
    The identity \eqref{eq: exactness quadrature Q} follows from \Cref{lem: quadrature_1} whereas the linear convergence \eqref{eq: assumed conv MFE} was shown in \Cref{thm:conv full BDM1-P0}. This allows us to invoke \Cref{thm:conv MS-MFE}.
\end{proof}

\section{A reducible scheme: \texorpdfstring{$\BDM1$-$\L1$}{BDM1-L1}}
\label{sec:reducible_scheme}

As noted in \Cref{sec: MS-MFE}, a key feature of the multipoint stress mixed finite element methods is that they can be reduced to systems involving only displacement and rotation degrees of freedom, by taking a Schur complement. In this section, we show how this can be improved by basing the construction of the MSMFE-1 method \cite[Sec.~4]{ambartsumyan2020multipoint}.
We choose the spaces as in \eqref{eq: Spaces BDM1-P0}, with the space for the rotation variable replaced by the nodal Lagrange finite element space, i.e.
\begin{align} \label{eq: Spaces BDM1-L1}
    \Sigma_h &\coloneqq \BDM1^d \cap \Sigma, &
    W_h &\coloneqq \BDM1^{k_d} \cap W, &
    U_h &\coloneqq \P0^d, &
    R_h &\coloneqq \L1^{k_d}.
\end{align}

\subsection{The mixed finite element method}

We emphasize that the above choice of spaces is not considered in \cite{boon2025mixed}, so the resulting mixed finite element method is a new method for the linear Cosserat equations. By choosing the smaller space $R_h$, property \eqref{eq:div-prop} does not hold and the results from \Cref{sec:simple_scheme} do not apply directly. Nevertheless, the stability and convergence analysis of the mixed finite element method \eqref{eq: system full-MFE} and the multipoint stress variant follow along the same lines, so we only highlight the differences when necessary.

\begin{theorem}[Stability] \label{thm: stability full BDM1-L1}
    With the finite element spaces from \eqref{eq: Spaces BDM1-L1}, Problem \eqref{eq: system full-MFE} admits a unique solution that satisfies
    \begin{align}
        \| \eta_h \|_{X_h} + \| v_h \|_Y \lesssim \| g \|_{X_h'} + \| f \|_{Y'}.
    \end{align}
\end{theorem}
\begin{proof}
    As in \Cref{thm: stability full BDM1-P0}, we have $\div \Sigma_h \subseteq U_h$. Moreover, the triplet $\Sigma \times U_h \times R_h$ in \eqref{eq: Spaces BDM1-L1} was shown to be elasticity-stable in \cite{cockburn2010new}, so \Cref{thm: stability full MFE} applies.
\end{proof}

\begin{theorem}[Convergence] \label{thm:convergence BDM1-L1}
    If the solution $(\eta, v)$ to \eqref{eq:cosserat_weak} is sufficiently regular and the finite element spaces are chosen as in \eqref{eq: Spaces BDM1-L1}, then the mixed finite element solution $(\eta_h, v_h)$ of \eqref{eq: system full-MFE} satisfies
    \begin{align}
        \| \eta_h - \eta \|_{X_h} + \| v_h - v \|_Y \lesssim h.
    \end{align}
\end{theorem}
\begin{proof}
  We follow the proof of \Cref{thm:conv full BDM1-P0}, using the stability from \Cref{thm: stability full BDM1-L1} instead of \Cref{thm: stability full BDM1-P0} until we reach inequality \eqref{eq: bound problematic term}. That bound does not hold in this case because property \eqref{eq:div-prop} does not hold and we therefore cannot subtract $\Pi_0 \ell$. Instead, we use the quadratic approximation property of $R_h = \L1^{k_d}$ and a discrete inverse inequality to derive
    \begin{align} \label{eq: bound problematic term L1}
        (\ell \div \omega_h', (I - \Pi_R) r)_\Omega
        &\le \| \ell \div \omega_h' \|_\Omega \| (I - \Pi_R) r \|_\Omega \nonumber \\
        &\lesssim \| \div \omega_h' \|_\Omega h^2 
        \lesssim h^{-1} \| \omega_h' \|_\Omega h^2
        = \| \omega_h' \|_\Omega h.
    \end{align}
    Since this is the same bound as in \eqref{eq: bound problematic term}, we may substitute it to obtain the analogous
    \begin{align} \label{eq: bound term 2 conv L1}
        \langle B \eta_h', (I - \Pi_Y) v \rangle \lesssim h \| \eta_h' \|_{X_h}.
    \end{align}
    The remainder of the proof is the same as in \Cref{thm:conv full BDM1-P0}, using the linear approximation properties of the interpolants $\pi_X$ and $\Pi_Y$ for the finite element spaces \eqref{eq: Spaces BDM1-P0}.
\end{proof}

\subsection{The multipoint stress mixed finite element method}

In the multipoint stress method, we aim to evaluate the balance of angular momentum using the quadrature rule from \Cref{lem: quadrature_1}. Recall from  \eqref{eq: momentum eqs} that this conservation law is given by
\begin{align} \label{eq: angular momentum rep}
    (\asym \sigma - \div \ell \omega, r')_\Omega 
    &= (f_\omega, r')_\Omega, &
    \forall r \in R.
\end{align}
We enforce this in the following weak sense:
\begin{align}
    (\asym \hat \sigma_h - \div \ell \hat \omega_h, r_h')_{Q_1} &= (\Pi_R f_\omega, r_h')_{Q_1}, &
    \forall r_h' &\in R_h,
\end{align}
in which $\Pi_R$ is the $L^2$ projection onto $R_h$, to ensure that the right-hand side is well-defined.
To incorporate this relation, we augment the operator $B$ and right-hand side $f$:
\begin{subequations} \label{eqs:B_h and f_h}
\begin{align}
    \langle B_h \eta_h, v_h' \rangle
    &\coloneqq -(\div \sigma_h, u_h')_\Omega
    + (\asym \sigma_h, r_h')_{Q_1}
    - (\div \ell \omega_h, r_h')_{Q_1}, &
    \forall (\eta_h, v_h') &\in X_h \times Y_h, \\
    \langle f_h, v_h' \rangle  &\coloneqq (f_\sigma, u_h')_\Omega + (\Pi_R f_\omega, r_h')_{Q_1}, &
    \forall v_h' &\in Y_h.
\end{align}
\end{subequations}

Using the operator $A_h$ from \eqref{eq: A_h BDM1}, the problem that defines the $\BDM1$-$\L1$ MS-MFE method is as follows: find $(\hat \eta_h, \hat v_h) \in X_h \times Y_h$ such that
\begin{align} \label{eq: system BDM L1}
    \langle A_h \hat \eta_h, \eta_h' \rangle
    - \langle B_h \eta_h', \hat v_h \rangle
    + \langle B_h \hat \eta_h, v_h' \rangle
    &= 
    \langle g, \eta_h' \rangle
    + \langle f_h, v_h' \rangle, &
    \forall (\eta_h', v_h') &\in X_h \times Y_h.
\end{align}

The continuity of $B_h$ in the $X_h$-norm of \eqref{eq: discrete norm} is not immediate. It therefore seems natural to analyze the problem in the following norm
\begin{align} \label{eq: discrete norm Q}
    \| \eta \|_{X_Q}^2 \coloneqq 
    \| \sigma \|_\Omega^2
    + \| \div \sigma \|_\Omega^2
    + \| \omega \|_\Omega^2
    + \| \Pi_R^Q \div \ell \omega \|_\Omega^2,
\end{align}
where $\Pi_R^Q$ is the orthogonal projection onto $R_h$ with respect to the $Q_1$-inner product. I.e.~for given $r \in R$, $\Pi_R^Q r \in R_h$ satisfies
\begin{align}
    (\Pi_R^Q r, r_h')_{Q_1} &= (\Pi_1 r, r_h')_{Q_1}, &
    \forall r_h' &\in R_h.
\end{align}
However, the norm \eqref{eq: discrete norm Q} is equivalent to \eqref{eq: discrete norm}, under an additional, but mild, regularity assumption on $\ell$.

\begin{lemma}[Norm equivalence] \label{lem:discrete norm equivalence}
    If $\ell|_{\Delta} \in W^{2, \infty}(\Delta)$ for each element $\Delta \in \Omega_h$, then 
    the norms from \eqref{eq: discrete norm} and \eqref{eq: discrete norm Q} are equivalent, i.e.
    \begin{align}
        \| \eta_h \|_{X_h} &\eqsim \| \eta_h \|_{X_Q}, &
        \forall \eta_h &\in X_h.
    \end{align}
\end{lemma}
\begin{proof}
    The only difference between the norms concerns the component $\omega_h$. We start with a triangle inequality
    \begin{align} \label{eq: triangle}
        \| \Pi_R^Q \div \ell \omega_h \|_\Omega &\le 
        \| \Pi_R \div \ell \omega_h \|_\Omega + \| (\Pi_R - \Pi_R^Q) \div \ell \omega_h \|_\Omega,
    \end{align}
    and focus on the second term. Since $\Pi_R^Q$ is a projection, we derive for $\phi \in R$:
    \begin{align}
        \| (\Pi_R - \Pi_R^Q) \phi \|_\Omega
        &= \| (\Pi_R - \Pi_R^Q) (I - \Pi_R) \phi \|_\Omega
        \lesssim \| (I - \Pi_R) \phi \|_\Omega.
    \end{align}

    We now continue by invoking the Poincar\'e inequality and using the fact that $\omega_h$ is piecewise linear:
    \begin{align}
        \| (\Pi_R - \Pi_R^Q) \div \ell \omega_h \|_\Omega
        \lesssim \| (I - \Pi_R) \div \ell \omega_h \|_\Omega 
        &\lesssim \sum_{\Delta \in \Omega_h} h \| \nabla (\div \ell \omega_h) \|_\Delta \nonumber \\
        &= \sum_{\Delta \in \Omega_h} h \| (\nabla^2 \ell) \omega_h + (\nabla \omega_h + (\div \omega_h) I) \nabla \ell \|_\Delta \nonumber \\
        &\lesssim \sum_{\Delta \in \Omega_h} h \| \omega_h \|_\Delta + h \| \nabla \omega_h \|_\Delta 
        \lesssim \| \omega_h \|_\Omega.
    \end{align}
    Here $\nabla^2 \ell$ is the Hessian of $\ell$, which is bounded on each element by assumption. 
    The combination of these bounds yields $\| \eta_h \|_{X_Q} \lesssim \| \eta_h \|_{X_h}$. The converse inequality follows analogously by interchanging terms in \eqref{eq: triangle}.
\end{proof}

Due to \Cref{lem:discrete norm equivalence}, we may continue the analysis of \eqref{eq: system BDM L1} in the $X_h$-norm.

\begin{theorem}[Stability] \label{thm:stability MSMFE BDML}
    If $\ell$ satisfies the assumption of \Cref{lem:discrete norm equivalence}, then the $\BDM1$-$\L1$ MS-MFE method is stable, i.e. Problem \eqref{eq: system BDM L1} admits a unique solution that satisfies
    \begin{align}
        \| \hat \eta_h \|_{X_h} + \| \hat v_h \|_Y \lesssim \| g \|_{X_h'} + \| f_h \|_{Y'}.
    \end{align}
\end{theorem}
\begin{proof}
    \Cref{thm: stability MS-MFE} does not apply because $B_h \ne B$. Nevertheless, we can follow the proof of \Cref{thm: stability full MFE}, using the equivalence of norms from \Cref{lem: quadrature_1,lem:discrete norm equivalence} when necessary. For instance, the continuity of $B_h$ is given by
    \begin{align}
        \langle B_h \eta_h, v_h \rangle
        &\le \| \div \sigma_h \|_\Omega \| u_h \|_\Omega
        + \| \asym \sigma_h \|_{Q_1} \| r_h \|_{Q_1}
        + \| \Pi_R^Q \div \ell \omega_h \|_{Q_1} \| r_h \|_{Q_1} \nonumber\\
        &\eqsim \| \div \sigma_h \|_\Omega \| u_h \|_\Omega
        + \| \asym \sigma_h \|_\Omega \| r_h \|_\Omega
        + \| \Pi_R^Q \div \ell \omega_h \|_\Omega \| r_h \|_\Omega \nonumber\\
        &\le \| \eta_h \|_{X_Q} \| v_h \|_Y
        \eqsim \| \eta_h \|_{X_h} \| v_h \|_Y, &
        \forall (\eta_h, v_h) &\in X_h \times Y_h.
    \end{align}
    The continuity of $A_h$ follows a similar calculation. 
    For the coercivity of $A_h$, we apply \eqref{eq: Ker B} with $\Pi_R$ replaced by $\Pi_R^Q$.
%
    It remains to show the inf-sup condition on $B_h$. From \cite{ambartsumyan2020multipoint}, we obtain the following inf-sup condition with a quadrature rule on the asymmetry term:
    \begin{align}        
        \inf_{(u_h, r_h) \in U_h \times R_h} \sup_{\sigma_h \in \Sigma_h}
        \frac{
            -(\div \sigma_h, u_h)_\Omega + (\asym \sigma_h, r_h)_{Q_1}
        }{\| (\sigma_h, 0) \|_X \| (u_h, r_h) \|_Y} \gtrsim 1.
    \end{align}
    The required inf-sup condition on $B_h$ now follows from \eqref{eq: infsup from elasticity}.
\end{proof}

By construction, the $\BDM1$-$\L1$ scheme is $\ell$-robust, since it reduces to the MSMFE-1 method of \cite[Sec.~4]{ambartsumyan2020multipoint}, in case of $\ell = 0$. Convergence is shown in the next theorem.

\begin{theorem}[Convergence] \label{thm:conv MSMFE BDML}
    If $\ell$ satisfies the assumptions of \Cref{lem:discrete norm equivalence}, 
    then the $\BDM1$-$\L1$ MS-MFE method converges linearly, i.e.
    \begin{align}
        \| \hat \eta_h - \eta \|_{X_h} + \| \hat v_h - v \|_Y \lesssim h.
    \end{align}
\end{theorem}
\begin{proof}
    Similar to \Cref{thm:conv MS-MFE}, we subtract \eqref{eq: system full-MFE} from \eqref{eq: system BDM L1}. We then note that $(\hat \eta_h - \eta_h, \hat v_h - v_h)$ solves \eqref{eq: system BDM L1} with right-hand side
    \begin{align}
        \langle g, \eta_h' \rangle 
        &\coloneqq 
        \langle (A - A_h) \eta_h, \eta_h' \rangle
        - \langle (B - B_h) \eta_h', v_h \rangle
        , &
        \langle f, v_h' \rangle 
        &\coloneqq 
        \langle (B - B_h) \eta_h, v_h' \rangle
        + (\Pi_R f_\omega, r_h')_{Q_1}
        - (f_\omega, r_h')_\Omega.
    \end{align}
    The first term of $g$ is bounded by \eqref{eq: bound diff A}, so we consider the second term instead. Letting $\Pi_0$ denote the projection onto the piecewise constants, we use \Cref{lem: quadrature_1} and the continuity of $B$ and $B_h$ to derive
    \begin{align} \label{eq: first-order term B_h}
        \langle (B - B_h) \eta_h', v_h \rangle
        &= 
        \langle (B - B_h) \eta_h', v_h - \Pi_0 v \rangle \nonumber \\
        &\lesssim 
        \| \eta_h' \|_{X_h} \| r_h - \Pi_0 r \|_\Omega
        \nonumber \\
        &\lesssim
        \| \eta_h' \|_{X_h} (\| r_h - r \|_\Omega + h \| r \|_{1, \Omega}).
    \end{align}
    
    Next, we consider the functional $f$, and note that
    \begin{align}
        \langle f, v_h' \rangle
        &=
        \left(
        (\asym \sigma_h, r_h')_\Omega
        - (\div \ell \omega_h, r_h')_\Omega
        - (f_\omega, r_h')_\Omega\right)
        -
        \left(
        (\asym \sigma_h, r_h')_{Q_1}
        - (\div \ell \omega_h, r_h')_{Q_1}
        - (\Pi_R f_\omega, r_h')_{Q_1} \right) \nonumber \\
        &= - (\asym \sigma_h, r_h')_{Q_1}
        + (\div \ell \omega_h, r_h')_{Q_1}
        + (\Pi_R f_\omega, r_h')_{Q_1},
    \end{align} 
    where the first terms vanish because $\eta_h$ solves \eqref{eq: system full-MFE}. Now, using the fact that the solution to \eqref{eq:cosserat_weak} satisfies \eqref{eq: angular momentum rep}, we have 
    $\Pi_R f_\omega = \Pi_R (\asym \sigma - \div \ell \omega) = \asym \Pi_{\L1} \sigma - \Pi_R \div \ell \omega$ with
    $\Pi_{\L1}$ the $L^2$ projection onto $\L1^{d \times d}$. Substituting this identity gives us
    \begin{align}
        \langle f, v_h' \rangle
        &=
        - (\asym (\sigma_h - \Pi_{\L1} \sigma), r_h')_{Q_1} + (\div \ell \omega_h - \Pi_R \div \ell \omega, r_h')_{Q_1} \nonumber \\
        &\lesssim
        \left(
        \| \sigma_h - \Pi_{\L1} \sigma \|_{Q_1} + \| \Pi_R^Q (\div \ell \omega_h - \Pi_R \div \ell \omega) \|_{Q_1}
        \right)
        \| r_h' \|_{Q_1} \nonumber \\
        &\lesssim
        \left(
        \| \sigma_h - \Pi_{\L1} \sigma \|_\Omega 
        + \| \Pi_R^Q \div \ell (\omega_h - \pi_W \omega) \|_\Omega
        + \| \Pi_R^Q (\div \ell \pi_W \omega - \Pi_R \div \ell \omega) \|_\Omega
        \right)
        \| r_h' \|_\Omega.
    \end{align}
    We now bound each of the three terms separately. The first is straightforward
    \begin{align}
        \| \sigma_h - \Pi_{\L1} \sigma \|_\Omega 
        &\le \| \sigma_h - \sigma \|_\Omega + \| (I - \Pi_{\L1}) \sigma \|_\Omega.
    \end{align}
    For the second term, we use \Cref{lem:discrete norm equivalence}:
    \begin{align}
        \| \Pi_R^Q \div \ell (\omega_h - \pi_W \omega) \|_\Omega
        &\lesssim 
        \| \omega_h - \pi_W \omega \|_\Omega +
        \| \Pi_R \div \ell (\omega_h - \pi_W \omega) \|_\Omega \nonumber \\
        &\lesssim 
        \| \omega_h - \omega \|_\Omega +
        \| (I - \pi_W) \omega \|_\Omega +
        \| \Pi_R \div \ell (\omega_h - \omega) \|_\Omega +
        \| \Pi_R \div \ell (I - \pi_W) \omega \|_\Omega \nonumber \\
        &\lesssim 
        \| \omega_h - \omega \|_\Omega +
        \| (I - \pi_W) \omega \|_\Omega +
        \| \Pi_R \div \ell (\omega_h - \omega) \|_\Omega +
        \| \div (I - \pi_W) \omega \|_\Omega.
    \end{align}
    The third and final term is bounded as
    \begin{align}
        \| \Pi_R^Q (\div \ell \pi_W \omega - \Pi_R \div \ell \omega) \|_\Omega
        &\lesssim 
        \| \div \ell (I - \pi_W) \omega \|_\Omega
        + \| (I - \Pi_R) \div \ell \omega) \|_\Omega \nonumber \\
        &\lesssim 
        \| (I - \pi_W) \omega \|_\Omega
        + \| \div (I - \pi_W) \omega \|_\Omega
        + \| (I - \Pi_R) \div \ell \omega) \|_\Omega.
    \end{align}

    The proof concludes by combining the stability bound from \Cref{thm:stability MSMFE BDML} with the linear convergence from \Cref{thm:convergence BDM1-L1}.
\end{proof}

The employed quadrature rule in $B_h$ allows us to further reduce the scheme to a system involving only the displacement variable, in regions of the domain where $\ell = 0$. We refer the interested reader to \cite{ambartsumyan2020multipoint} for details concerning this additional reduction.

\section{A higher-order scheme with continuous rotations: \texorpdfstring{$\RT1$-$\L1$}{RT1-L1}}
\label{sec:RT1-L1}

The third multipoint stress mixed finite element scheme we consider aims to achieve second-order convergence by employing finite element spaces of higher order. For this method, we consider the following spaces, as outlined previously in \Cref{tab: summary}:
\begin{align} \label{eq: Spaces RT1-L1}
    \Sigma_h &\coloneqq \RT1^d \cap \Sigma, &
    W_h &\coloneqq \RT1^{k_d} \cap W, &
    U_h &\coloneqq \P1^d, &
    R_h &\coloneqq \L1^{k_d}.
\end{align}

The idea behind this method is to apply the quadrature rule from \Cref{lem: quadrature_2} and obtain second-order convergence. However, since the inclusion \eqref{eq:div-prop} does not hold for the above choice of spaces, the accuracy of the method for $\ell > 0$ is reduced to first-order. We do obtain second-order accuracy for the case of linear elasticity with $\ell = 0$. In the analysis will use the same strategy as for the previous two methods, and only highlight the differences, where necessary.

\subsection{The mixed finite element method}

As in \Cref{sec:reducible_scheme}, the above choice of spaces is not considered in \cite{boon2025mixed}, so the resulting mixed finite element method is a new method for the linear Cosserat equations. Moreover, to our knowledge, the spaces of \eqref{eq: Spaces RT1-L1} have not previously been analyzed as a discretization of elasticity with weak symmetry. We therefore first show the relevant inf-sup stability by employing the following lemma.

\begin{lemma} \label{lem:elasticity_triple_constr}
    Let $\Sigma_h \times U_h$ be Darcy-stable. If a finite element space $\Theta_h$ exists such that $\nabla \times \Theta_h \subseteq \Sigma_h$ and $(S\Theta_h) \times R_h$ is Stokes-stable,
    then $\Sigma_h \times U_h \times R_h$ is elasticity-stable in the sense of \Cref{def:elasticity triplet}.
\end{lemma}
\begin{proof}
    See e.g. \cite[Thm. 4.2]{ambartsumyan2020multipoint} or \cite[Thm. 1]{lee2016towards}.
\end{proof}

\begin{lemma} \label{lem: RT1-L1 elasticity stable}
    The triplet $\RT1^d \times \P1^d \times \L1^{k_d}$ is elasticity-stable.
\end{lemma}
\begin{proof}
    The Taylor-Hood pair $\L2^d \times \L1$ is Stokes-stable \cite[Sec.~8.8]{boffi2013mixed}. Let $\Theta_h \coloneqq \L2^{d \times k_d}$, then $S\Theta_h = \Theta_h$ and so $S(\Theta_h) \times R_h = \L2^{d \times k_d} \times \L1^{k_d}$ is Stokes-stable. Moreover $\nabla \times \Theta_h \subseteq \RT1^d = \Sigma_h$ and therefore \Cref{lem:elasticity_triple_constr} applies.
\end{proof}

\begin{theorem}[Stability] \label{thm:stability RT1L1}
    If the discrete spaces are given by \eqref{eq: Spaces RT1-L1}, then Problem \eqref{eq: system full-MFE} admits a unique solution that satisfies
    \begin{align}
        \| \eta_h \|_{X_h} + \| v_h \|_Y \lesssim \| g \|_{X_h'} + \| f \|_{Y'}.
    \end{align}
\end{theorem}
\begin{proof}
    The inclusion $\div \Sigma_h \subseteq U_h$ and \Cref{lem: RT1-L1 elasticity stable} allow us to invoke \Cref{thm: stability full MFE}.
\end{proof}

\begin{theorem}[Convergence] \label{thm:conv full RT1-L1}
    If the true solution $(\eta, v)$ is sufficiently regular, and the discrete spaces are given by \eqref{eq: Spaces RT1-L1}, then the mixed finite element method converges as
    \begin{align}
        \| \eta_h - \eta \|_{X_h} + \| v_h - v \|_Y \lesssim h.
    \end{align}
    Moreover, if $\ell = 0$, then
    \begin{align} \label{eq: quadratic RT1-L1}
        \| \eta_h - \eta \|_{X_h} + \| v_h - v \|_Y \lesssim h^2.
    \end{align}
\end{theorem}
\begin{proof}
    We follow the proofs of \Cref{thm:conv full BDM1-P0,thm:convergence BDM1-L1}. In this case, we employ the interpolation and projection operators $\pi_X$ and $\Pi_Y$, which are second-order accurate for the spaces \eqref{eq: Spaces RT1-L1}. With these operators, the calculations \eqref{eq: bound term 1 conv}, \eqref{eq: bound term 2 conv L1}, and \eqref{eq: bound term 3 conv} provide the bounds
    \begin{align}\label{eq:bounds-RT1-L1}
        \langle A (I - \pi_X) \eta, \eta_h' \rangle 
        &\lesssim h^2 \| \eta_h' \|_{X_h}, &
        \langle B \eta_h', (I - \Pi_Y) v \rangle 
        &\lesssim h \| \eta_h' \|_{X_h}, &
        \langle B (I - \pi_X) \eta, v_h' \rangle
        &\lesssim h^2 \| v_h' \|_Y.
    \end{align}
    Note that the second term is only first order accurate because of \eqref{eq: bound problematic term L1}. This is due to the fact that the inclusion \eqref{eq:div-prop} does not hold.
If, however, $\ell = 0$, then \eqref{eq: bound problematic term L1} is unnecessary and the source of this error disappears.
    The remainder of the proof is the same as for \Cref{thm:conv full BDM1-P0}, using the quadratic interpolation estimates.
\end{proof}

\begin{remark}
When $\ell = 0$, method \eqref{eq: system full-MFE} with choice of discrete spaces given in \eqref{eq: Spaces RT1-L1} is a new second-order mixed finite element method for linear elasticity with weak stress symmetry.
\end{remark}

\subsection{The multipoint stress mixed finite element method}

To formulate the multipoint stress mixed finite element method, we employ the $Q_2$-quadrature rule from \Cref{lem: quadrature_2} and approximate the bilinear form $A$ as:
\begin{align} \label{eq: A_h RT1}
    \langle A_h \eta_h , \eta_h'  \rangle
    &\coloneqq (\mathcal{A}_\sigma \sigma_h, \sigma_h')_{Q_2} + (\mathcal{A}_\omega \omega_h, \omega_h')_{Q_2}.
\end{align}
We then define the $\RT1$-$\L1$ multipoint stress mixed finite element method by the following problem: find $(\hat \eta_h, \hat v_h) \in X_h \times Y_h$ such that
\begin{align} \label{eq: system RT1 L1}
    \langle A_h \hat \eta_h, \eta_h' \rangle
    - \langle B \eta_h', \hat v_h \rangle
    + \langle B \hat \eta_h, v_h' \rangle
    &= 
    \langle g, \eta_h' \rangle
    + \langle f, v_h' \rangle, &
    \forall (\eta_h', v_h') &\in X_h \times Y_h.
\end{align}

\begin{remark} \label{rem: why not Bh}
    It is possible to approximate $B$ by $B_h$ from \eqref{eqs:B_h and f_h}. However, the $Q_1$ quadrature rule introduces a first order error term, cf.~\eqref{eq: first-order term B_h} in \Cref{thm:convergence BDM1-L1}, that we aim to avoid. If we used the $Q_2$ quadrature rule instead, then the rotation space $\L1$ does not localize and can therefore not be eliminated in the elasticity limit. The introduction of a quadrature rule in $B$ would therefore not lead to any practical benefits.
\end{remark}

\begin{theorem} \label{thm:stab/conv MSMFE RT1L1}
    The $\RT1$-$\L1$ MS-MFE method \eqref{eq: system RT1 L1} is stable and linearly convergent, i.e.
    \begin{align}\label{eq:RT1-L1-h}
        \| \hat \eta_h \|_{X_h} + \| \hat v_h \|_Y &\lesssim \| g \|_{X_h'} + \| f \|_{Y'}, &
        \| \hat \eta_h - \eta \|_{X_h} + \| \hat v_h - v \|_Y &\lesssim h.
    \end{align}
    Moreover, if $\ell = 0$, then the following quadratic convergence estimate holds:
 \begin{align}\label{eq:RT1-L1-h2}
     \| \hat \eta_h - \eta \|_{X_h}
     + \| \Pi_0 (\hat u_h - u) \|_\Omega
     + \| \hat r_h - r \|_\Omega
     \lesssim h^2.
 \end{align}
\end{theorem}
\begin{proof}
  Since $\RT1 \subset \P2^d$, \Cref{lem: quadrature_2} provides the norm equivalence \eqref{eq: norm equivalence Q} and, in turn, the stability follows from \Cref{thm: stability MS-MFE}. The convergence estimate in \eqref{eq:RT1-L1-h} follows from \Cref{thm:conv MS-MFE}, using the exactness of the quadrature rule from \Cref{lem: quadrature_2} and the linear convergence from \Cref{thm:conv full RT1-L1}.

We proceed with the second-order estimate \eqref{eq:RT1-L1-h2} when $\ell = 0$.
We first consider the components $\| \div (\hat \sigma_h - \sigma) \|_\Omega$ and $\| \hat \sigma_h - \sigma \|_\Omega$, and employ the shorter notation $\tilde \sigma_h \coloneqq \hat \sigma_h - \sigma_h \in \Sigma_h$. By \eqref{eq: system full-MFE} and \eqref{eq: system RT1 L1}, we note that
    \begin{align}
        (\div \tilde \sigma_h, u_h')_\Omega &= 0, &
        \forall u_h' &\in U_h.
    \end{align}
Now, $\div \Sigma_h = U_h$ implies that $\div \tilde \sigma_h = 0$. The estimate on the divergence term now follows from \Cref{thm:conv full RT1-L1} since
\begin{equation}\label{eq:div-h2}
  \| \div (\hat \sigma_h - \sigma) \|_\Omega
  = \| \div (\sigma_h - \sigma) \|_\Omega \lesssim h^2.
\end{equation}
We continue by bounding $\|\tilde \sigma_h\|_\Omega$. To shorten notation further, let $\tilde \eta \coloneqq (\tilde \sigma_h, 0)$ so that
    \begin{align} \label{eq: bound tilde sigma 1}
        \| \tilde \sigma_h \|_\Omega^2 
        \eqsim \| \tilde \sigma_h \|_{Q_2}^2 
        \eqsim \langle A_h \tilde \eta_h, \tilde \eta_h \rangle
        = \langle A_h (\hat \eta_h - \eta_h), \tilde \eta_h \rangle
        = \langle (A - A_h) \eta_h, \tilde \eta_h \rangle
    \end{align}
    where the final equality is due to the error equations \eqref{eq: error equations} with $v_h' = \hat v_h - v_h$ and $\eta_h' = \tilde \eta_h$. 
    We now follow the reasoning of \cite{egger2020second}. Since $\tilde \sigma_h$ is solenoidal, \cite[Cor.~2.3.1]{boffi2013mixed} implies that $\tilde \sigma_h \in \BDM1^d \subset \P1^{d \times d}$. Hence, by \Cref{lem: quadrature_2}, we have $\langle (A - A_h) \Pi_1 \eta, \tilde \eta_h \rangle = 0$ with $\Pi_1$ the $L^2$ projection onto the discontinuous, piecewise linears. This allows us to continue the bound \eqref{eq: bound tilde sigma 1} as
\begin{align} \label{eq: bound tilde sigma 2}
        \langle (A - A_h) \eta_h, \tilde \eta_h \rangle
        = \langle (A - A_h) (\eta_h - \Pi_1 \eta), \tilde \eta_h \rangle
        &\lesssim \| \sigma_h - \Pi_1 \sigma \|_\Omega \| \tilde \sigma_h \|_\Omega \nonumber \\
        &\le (\| \sigma_h - \sigma \|_\Omega + \| (I - \Pi_1) \sigma \|_\Omega) \| \tilde \sigma_h \|_\Omega.
\end{align}
Combining \eqref{eq: bound tilde sigma 1} and \eqref{eq: bound tilde sigma 2} with the approximation properties of $\Pi_1$ and applying the quadratic convergence estimate \eqref{eq: quadratic RT1-L1} from \Cref{thm:conv full RT1-L1}, we obtain
\begin{equation}\label{eq:sigma-h2}
\| \hat \sigma_h - \sigma \|_\Omega \lesssim h^2.
\end{equation}
We next note that if $\ell = 0$, then $\omega = \hat \omega_h = 0$ and thus $\| \hat \eta_h - \eta \|_{X_h} \eqsim \| \hat \sigma_h - \sigma \|_\Omega + \| \div (\hat \sigma_h - \sigma) \|_\Omega$, which, combined with \eqref{eq:div-h2} and \eqref{eq:sigma-h2}, results in
\begin{equation}\label{eq:eta-h2}
\| \hat \eta_h - \eta \|_{X_h} \lesssim h^2.
\end{equation}

It remains to establish the bound on the displacement and rotation variables. 
For that, we consider the error equation in $\Sigma_h$, given by
\begin{align}
    (\mathcal{A}_\sigma \hat \sigma_h, \sigma_h')_{Q_2} 
    - (\mathcal{A}_\sigma \sigma_h, \sigma_h')_\Omega
    + (\div \sigma_h', \tilde u_h)_\Omega
    - (\asym \sigma_h', \tilde r_h)_\Omega &= 0 \quad \forall \sigma_h' \in \Sigma_h,
\end{align}
We rearrange the terms and introduce $\tilde \sigma_h$ to obtain
\begin{align} \label{eq: error sigma}
    - (\div \sigma_h', \tilde u_h)_\Omega
    + (\asym \sigma_h', \tilde r_h)_\Omega
    = (\mathcal{A}_\sigma \tilde \sigma_h, \sigma_h')_{Q_2} 
    + \left[(\mathcal{A}_\sigma \sigma_h, \sigma_h')_{Q_2} - (\mathcal{A}_\sigma \sigma_h, \sigma_h')_\Omega \right].
\end{align}
The first term on the right-hand side can be easily bounded. For the terms in the square brackets, we aim to use the exactness of the quadrature rule from \Cref{lem: quadrature_2}. Since $\sigma_h' \in \P2^{d \times d}$, we may subtract $\Pi_0 \sigma$ as in \eqref{eq: bound diff A}, but this would lead to a first-order estimate. Instead, we construct a piecewise linear test functions from $\BDM1^d \subset \Sigma_h$ and
use $\Pi_1 \sigma$ as follows. The elasticity stability ensures that for given $\tilde r_h \in R_h$, a $\tau_h^r \in \Sigma_h$ exists such that
\begin{align} \label{eq: tau r}
    \div \tau_h^r &= 0, &
    \Pi_R \asym \tau_h^r &= \tilde r_h, &
    \| \tau_h^r \|_\Omega &\lesssim \| \tilde r_h \|_\Omega.
\end{align}
Its construction uses the auxiliary space $\Theta_h$ from \Cref{lem: RT1-L1 elasticity stable}. In particular, the Stokes-stability with $R_h$ provides $\theta_h \in \Theta_h$ that satisfies $\Pi_R \div S \theta_h = r_h$ and $\|\theta_h\|_{H^1(\Omega)} \lesssim \|\tilde r_h \|_\Omega$. The test function is then given by $\tau_h^r = \nabla \times \theta_h$, which satisfies \eqref{eq: tau r} because of the identity \eqref{eq: asym identity}.

Since $\tau_h^r$ is solenoidal, we use \cite[Cor.~2.3.1]{boffi2013mixed} again to conclude that $\tau_h^r \in \BDM1^d$. Substituting this test function in \eqref{eq: error sigma}, we derive
\begin{align} \label{eq: BDM1 trick for v}
    \| \tilde r_h \|_\Omega^2
    &=
    (\mathcal{A}_\sigma \tilde \sigma_h, \tau_h^r)_{Q_2} 
    + \left[(\mathcal{A}_\sigma \sigma_h, \tau_h^r)_{Q_2} - (\mathcal{A}_\sigma \sigma_h, \tau_h^r)_\Omega \right]
     \nonumber \\
    &=
    (\mathcal{A}_\sigma \tilde \sigma_h, \tau_h^r)_{Q_2} 
    + \left[(\mathcal{A}_\sigma (\sigma_h - \Pi_1 \sigma), \tau_h^r)_{Q_2} - (\mathcal{A}_\sigma (\sigma_h - \Pi_1 \sigma), \tau_h^r)_\Omega \right]
     \nonumber \\
    &\lesssim
    (\| \tilde \sigma_h \|_\Omega 
        + \| \sigma_h - \Pi_1 \sigma \|_\Omega) \| \tau_h^r \|_\Omega \nonumber \\
    &\lesssim h^2 \| \tilde r_h \|_\Omega.
\end{align}

Finally, we consider the displacement variable. The inf-sup stability of $\BDM1 \times \P0$ allows us to construct $\tau_h^u \in \BDM1^d$ that satisfies
\begin{align}
    \div \tau_h^u &= -\Pi_0 \tilde u_h, &
    \| \tau_h^u \|_\Omega &\lesssim \| \Pi_0 \tilde u_h \|_\Omega.
\end{align}
Substituting this test function in \eqref{eq: error sigma} and using the same steps as \eqref{eq: BDM1 trick for v}, we obtain
\begin{align}
    \| \Pi_0 \tilde u_h \|^2 &= 
    - (\asym \tau_h^u, \tilde r_h)_\Omega
    + (\mathcal{A}_\sigma \tilde \sigma_h, \tau_h^u)_{Q_2} 
    + \left[(\mathcal{A}_\sigma \sigma_h, \tau_h^u)_{Q_2} 
    - (\mathcal{A}_\sigma \sigma_h, \tau_h^u)_\Omega \right]
     \nonumber \\
    &\lesssim  
    (\| \tilde r_h \|_\Omega + \| \tilde \sigma_h \|_\Omega 
        + \| \sigma_h - \Pi_1 \sigma \|_\Omega)\| \tau_h^u \|_\Omega \nonumber \\
    &\lesssim 
    h^2 \| \Pi_0 \tilde u_h \|_\Omega.
\end{align}

Combined with \eqref{eq: quadratic RT1-L1}, this gives us the quadratic convergence estimate
\begin{align} \label{eq:v-h2}
    \| \Pi_0 (\hat u_h - u) \|_\Omega + \| \hat r - r \|_\Omega \lesssim h^2.
\end{align}
The proof of \eqref{eq:RT1-L1-h2} is completed by combining \eqref{eq:eta-h2} and \eqref{eq:v-h2}.
\end{proof}

\begin{remark}
  When $\ell = 0$, method \eqref{eq: system RT1 L1} is a new second-order multipoint stress mixed finite element method for linear elasticity with weak stress symmetry. The quadratic convergence of the displacement is only proven for the mean per element, similar to \cite[Lem.~4.4]{egger2020second}.
\end{remark}

\section{A higher-order scheme with discontinuous rotations: \texorpdfstring{$\RT1$-$\P1$}{RT1-P1}}
\label{sec:RT1-P1}

The fourth and final multipoint stress mixed finite element scheme we consider employs quadratic finite elements for the stress variables and discontinuous linear polynomials for the rotations:
\begin{align} \label{eq: Spaces RT1-P1}
    \Sigma_h &\coloneqq \RT1^d \cap \Sigma, &
    W_h &\coloneqq \RT1^{k_d} \cap W, &
    U_h &\coloneqq \P1^d, &
    R_h &\coloneqq \P1^{k_d}.
\end{align} 

The triplet $\Sigma_h \times U_h \times R_h$ in \eqref{eq: Spaces RT1-P1} was shown to be elasticity-stable in \cite[Sec. 4.2.3]{lee2016towards} on barycentrically subdivided grids. We emphasize that property \eqref{eq:div-prop} holds with this choice of spaces. As a result, unlike \Cref{sec:RT1-L1}, the mixed finite element method for the Cosserat system is second-order accurate, and the multipoint stress method exhibits second-order convergence for the Cauchy stress and rotation variables. The analysis follows by the same steps as in the previous sections.

\begin{remark} \label{rem:barycentric}
    The restriction on the grid forms a notable drawback of the method. In particular, for a given 3D simplicial grid, such a subdivision increases the number of cells by a factor 4 and, in turn, the space $Y_h$ contains 96 degrees of freedom per element of the original grid.
\end{remark}

\subsection{The mixed finite element method}

As in \Cref{sec:reducible_scheme,sec:RT1-L1}, the above choice of spaces is not considered in \cite{boon2025mixed}, so the resulting mixed finite element method is a new method for the linear Cosserat equations.

\begin{theorem} \label{thm: stab/conv full RT1-P1}
    For the discrete spaces given by \eqref{eq: Spaces RT1-P1}, problem \eqref{eq: system full-MFE} admits a unique solution that satisfies
    \begin{align}
        \| \eta_h \|_{X_h} + \| v_h \|_Y \lesssim \| g \|_{X_h'} + \| f \|_{Y'}.
    \end{align}
    Moreover, if the solution $(\eta, v)$ to \eqref{eq:cosserat_weak} is sufficiently regular, then the mixed finite element solution satisfies
    \begin{align}\label{eq:conv-full-RT1-P1}
        \| \eta_h - \eta \|_{X_h} + \| v_h - v \|_Y \lesssim h^2.
    \end{align}
\end{theorem}
\begin{proof}
    The stability follows from \Cref{thm: stability full MFE}, using the elasticity stability of the triplet from \cite[Sec.~4.2.3]{lee2016towards}.
    We derive the convergence estimate by following the proof of \Cref{thm:conv full BDM1-P0}, see also the proof of \Cref{thm:conv full RT1-L1}. Since the interpolation and projection operators $\pi_X$ and $\Pi_Y$ are second-order accurate for the spaces \eqref{eq: Spaces RT1-P1}, i.e $\| (I - \pi_X) \eta \|_{X_h} \lesssim h^2$ and $\| (I - \Pi_Y) v \|_Y \lesssim h^2$, the calculations \eqref{eq: bound term 1 conv}, \eqref{eq: bound term 2 conv}, and \eqref{eq: bound term 3 conv} provide the bounds
\begin{align}\label{eq:bounds-RT1-P1}
        \langle A (I - \pi_X) \eta, \eta_h' \rangle 
        &\lesssim h^2 \| \eta_h' \|_{X_h}, &
        \langle B \eta_h', (I - \Pi_Y) v \rangle 
        &\lesssim h^2 \| \eta_h' \|_{X_h}, &
        \langle B (I - \pi_X) \eta, v_h' \rangle
        &\lesssim h^2 \| v_h' \|_Y.
\end{align}
We emphasize the second-order bound for the second term above, unlike the first-order bound for this term in the proof of \Cref{thm:conv full RT1-L1}, cf. \eqref{eq:bounds-RT1-L1}. The reason is that, since \eqref{eq:div-prop} holds, we can use the argument in \eqref{eq: bound problematic term} to obtain a second-order bound in \eqref{eq: bound term 2 conv}. Bound \eqref{eq:conv-full-RT1-P1} follows from the proof of \Cref{thm:conv full BDM1-P0} using the estimates \eqref{eq:bounds-RT1-P1}.  
\end{proof}

\subsection{The multipoint stress mixed finite element method}

The $\RT1$-$\P1$ multipoint stress mixed finite element method employs the bilinear form $A_h$ from \eqref{eq: A_h RT1} and finds $(\hat \eta_h, \hat v_h) \in X_h \times Y_h$ such that
\begin{align} \label{eq: system RT1 P1}
    \langle A_h \hat \eta_h, \eta_h' \rangle
    - \langle B \eta_h', \hat v_h \rangle
    + \langle B \hat \eta_h, v_h' \rangle
    &= 
    \langle g, \eta_h' \rangle
    + \langle f, v_h' \rangle, &
    \forall (\eta_h', v_h') &\in X_h \times Y_h.
\end{align}

We now prove stability and error bounds for this method. The result is stronger than \Cref{thm:stab/conv MSMFE RT1L1}, since we obtain second-order accuracy for the Cauchy stress in the general case $\ell \ge 0$.

\begin{theorem} \label{thm:stab/conv MSMFE RT1P1}
The $\RT1$-$\P1$ MS-MFE method \eqref{eq: system RT1 P1} has a unique solution that satisfies
 \begin{align}
   \| \hat \eta_h \|_{X_h} + \| \hat v_h \|_Y &\lesssim \| g \|_{X_h'} + \| f \|_{Y'}.
 \end{align}
Moreover, if the solution $(\eta, v)$ to \eqref{eq:cosserat_weak} is sufficiently regular, then
   \begin{align}\label{eq:conv-RT1-P1-h2}
        \| \hat \sigma_h - \sigma \|_\Omega
        + \| \div (\hat \sigma_h - \sigma) \|_\Omega 
        + \| \Pi_R \div \ell (\hat \omega_h - \omega) \|_\Omega
        + \| \Pi_0 (\hat u_h - u) \|_\Omega 
        + \| \hat r - r \|_\Omega 
        \lesssim h^2.
   \end{align}
   and
   \begin{align}\label{eq:conv-RT1-P1-h}
        \|\hat \omega_h - \omega\|_\Omega 
        + \| \hat u_h - u \|_\Omega
        \lesssim h.
   \end{align}
\end{theorem}
\begin{proof}
\Cref{lem: quadrature_2} suffices to invoke \Cref{thm: stability MS-MFE} and obtain stability. With the addition of \Cref{thm: stab/conv full RT1-P1}, \Cref{thm:conv MS-MFE} provides the linear convergence bound \eqref{eq:conv-RT1-P1-h}.

The second-order convergence of $\| \hat \sigma_h - \sigma \|_\Omega$ and $\| \div (\hat \sigma_h - \sigma) \|_\Omega$ in \eqref{eq:conv-RT1-P1-h2}
follows by the same arguments as in \Cref{thm:stab/conv MSMFE RT1L1}, using the quadratic convergence from \Cref{thm: stab/conv full RT1-P1}.
Similarly, the quadratic estimate on $\| \Pi_0 (\hat u_h - u) \|_\Omega$ and $\| \hat r - r \|_\Omega$ follows from the arguments from \Cref{thm:stab/conv MSMFE RT1L1} with the existence of $\Theta_h$ presented in \cite[Sec.~4.2.3]{lee2016towards}.

For the third term in \eqref{eq:conv-RT1-P1-h2}, we note that $\langle B (\hat \eta_h - \eta), v_h' \rangle = 0$ for all $v_h' \in Y_h$, which implies
    \begin{align}
        \| \Pi_R \div \ell (\hat \omega_h - \omega) \|_\Omega
        = \| \Pi_R \asym (\hat \sigma_h - \sigma) \|_\Omega
        \lesssim \| \hat \sigma_h - \sigma \|_\Omega
        \lesssim h^2.
    \end{align}
\end{proof}

\begin{remark}
The reason for obtaining only first-order bound for $\|\hat \omega_h - \omega\|_\Omega$ in \eqref{eq:conv-RT1-P1-h} is that $\hat \omega_h - \omega_h$ is not divergence-free, so it is not in $\P1^{d \times d}$, hence the argument used for $\hat \sigma_h - \sigma_h$ in \eqref{eq: bound tilde sigma 2} cannot be applied.
\end{remark}

\section{Numerical results} 
\label{sec:num}

We validate the proposed discretization schemes by considering two numerical test cases, one in 2D and one in 3D, following a set-up similar to \cite{boon2025mixed}.
The computational domain is given by the unit square, respectively cube, $\Omega \coloneqq (0, 1)^d$ for $d=2, 3$.
The material parameters are set as $\mu_\sigma = \mu_\omega = 1$, $\mu_\sigma^c = \mu_\omega^c = 0.1$, and $\lambda_\sigma = \lambda_\omega = 1$.
We prescribe the analytical displacement and rotation solution as:
\begin{subequations}
\begin{align}
    u(x) &= \begin{cases}
        \sum_{i=1}^{d}{x_{i+1}\left(1-x_{i+1}\right)\sin{\left(\pi x_i\right)}\bm{e}_i}, & d = 2, \\
        \sum_{i=1}^{d}{x_{i+1}\left(1-x_{i+1}\right)x_{i-1}\left(1-x_{i-1}\right)\sin{\left(\pi x_i\right)}\bm{e}_i}, & d = 3, \\
    \end{cases} \\
    r(x) &= 
    \begin{cases}    
        \sin{\left(\pi x_1\right)}\sin{\left(\pi x_2\right)}, & d = 2, \\
        \sum_{i=1}^{d}{x_i\left(1-x_i\right)\sin{\left(\pi x_{i+1}\right)}\sin{\left(\pi x_{i-1}\right)}\bm{e}_i}, & d = 3,
    \end{cases}
\end{align}
\end{subequations}
in which $x = [x_1, \ldots, x_d]$ and the indices $i$ are understood modulo $d$. Moreover $\bm{e}_i$ is the $i$-th canonical basis vector of $\mathbb{R}^d$.
This choice of rotation and displacements allows us to set homogeneous natural boundary conditions on $\partial_n \Omega = \partial \Omega$. We consider two variants concerning $\ell$ by either setting $\ell = 1$ or $\ell = \varpi$ with:
\begin{align}
    \varpi(x) = 
    \begin{cases}
        0, & 0 \le x_1 < \frac13, \\
        \sin^2\left(\frac\pi2(3 x_1-1)\right), & \frac13 \le x_1 < \frac23, \\
        1, & \frac23 \le x_1 \le 1.\\
    \end{cases}
\end{align}
Note that $\varpi \in H^2(\Omega)$ represents a smooth transition function between a linearly elastic material where $\ell = 0$ and a Cosserat material where $\ell = 1$. The computational grids are chosen to conform to the planes at $x_1 = \frac13$ and $x_1 = \frac23$.

By setting the right-hand side terms $g_\sigma$ and $g_\omega$ to zero, we derive the stresses $\sigma$ and $\omega$ according to \eqref{eq: constit laws}. In turn, we derive the corresponding right-hand side terms $f_\sigma$ and $f_\omega$ in \eqref{eq: momentum eqs} analytically.
For each method, we then compare the performance of the multipoint stress (MS-MFE) method with the corresponding full mixed finite element (MFE) method. We evaluate each method by computing the $L^2$-error with respect to the known solution.

\begin{remark}
    If $\ell$ is chosen to be a piecewise linear transition between zero and one, then the solution constructed in this way is not sufficiently regular to satisfy the assumptions for the quadratic convergence estimates. As a result, we observed that all methods converge only linearly. These results are omitted for brevity.
\end{remark}

All results are computed with the libraries PorePy \cite{Keilegavlen2020} and PyGeoN \cite{pygeon}, using direct solvers from UMFPACK \cite{Davis2004} for the MFE in 2D and Cholesky decomposition from CHOLMOD \cite{Chen2008} for the MS-MFE. 
In 3D, the MFE systems are too computationally demanding to solve directly. We therefore apply GMRes from SciPy \cite{virtanen2020scipy}, using the MS-MFE solver as a preconditioner, until a relative residual of $10^{-6}$ is reached.

The two-dimensional meshes are unstructured and generated using Gmsh \cite{Geuzaine2009}, whereas we choose structured tetrahedral grids in 3D. 
The run-scripts for the numerical tests are publicly available at \url{https://github.com/compgeo-mox/cosserat}.
For ease of reference, we summarize the observed and theoretical convergence rates in \Cref{tab: summary_conv}.
\begin{table}[ht]
    \caption{Observed and predicted (in parentheses) convergence rates of the proposed multipoint stress mixed finite element methods.}
    \label{tab: summary_conv}
    \centering

    \begin{tabular}{cc|cccc}
    \hline

    \hline
     Name & Table & Order($\sigma$) & Order($\omega$) & Order($u$) & Order($r$) \\
    \hline
        $\BDM1$-$\P0$ & \ref{tab: BDM1-P0-2d} &
        1 (1) & 1 (1) & 1 (1) & 1 (1) \\
        $\BDM1$-$\L1$ & \ref{tab: BDM1-L1-2d} &
        1 (1) & 1 (1) & 1 (1) & 1-2 (1) \\
        $\RT1$-$\L1$ & \ref{tab: RT1-L1-2d} &
        2 (1) & $\sim$1 (1) & 2 (1) & 1-2 (1) \\
        $\RT1$-$\P1$ & \ref{tab: RT1-P1-2d} &
        2 (2) & 2 (1) & 2 (1) & $\sim$2 (2) \\
    \hline
    \end{tabular}
\end{table}

\subsection{The simple scheme \texorpdfstring{$\BDM1$-$\P0$}{BDM1-P0}}
\label{sub:the_simple_scheme}
\begin{table}[ht]
    \centering
    \caption{Convergence results for the $\BDM1$-$\P0$ mixed finite element methods from \Cref{sec:simple_scheme}.}
    \label{tab: BDM1-P0-2d}
    {\footnotesize
    \begin{tabular}{|c|c|cc|cc|cc|cc|c|c|}
    \hline
    2D $(\ell = 1)$ & $h$ & Error$(\sigma)$ & Order & Error$(\omega)$  & Order & Error$(u)$  & Order & Error$(r)$  & Order & DoF \\
    \hline
    \multirow{4}{*}{\rotatebox[origin=c]{90}{MFE}}
& 7.85e-02 & 6.30e-03 &    - & 3.52e-03 &    - & 1.57e-01 &    - & 1.11e-01 &    - & 7.72e+03 \\ 
& 4.25e-02 & 2.73e-03 & 1.36 & 8.86e-04 & 2.25 & 7.98e-02 & 1.10 & 5.63e-02 & 1.11 & 2.94e+04 \\ 
& 2.05e-02 & 1.30e-03 & 1.02 & 2.17e-04 & 1.93 & 3.97e-02 & 0.96 & 2.80e-02 & 0.96 & 1.19e+05 \\ 
& 1.04e-02 & 6.45e-04 & 1.02 & 5.46e-05 & 2.02 & 2.00e-02 & 1.01 & 1.41e-02 & 1.01 & 4.65e+05 \\
    \hline
    \multirow{4}{*}{\rotatebox[origin=c]{90}{MS-MFE}}
& 7.85e-02 & 2.45e-02 &    - & 3.32e-02 &    - & 1.58e-01 &    - & 1.12e-01 &    - & 1.88e+03 \\ 
& 4.25e-02 & 1.23e-02 & 1.12 & 1.71e-02 & 1.09 & 7.99e-02 & 1.11 & 5.64e-02 & 1.11 & 7.25e+03 \\ 
& 2.05e-02 & 6.07e-03 & 0.97 & 8.43e-03 & 0.97 & 3.97e-02 & 0.96 & 2.80e-02 & 0.96 & 2.95e+04 \\ 
& 1.04e-02 & 3.03e-03 & 1.02 & 4.24e-03 & 1.01 & 2.00e-02 & 1.01 & 1.41e-02 & 1.01 & 1.16e+05 \\
    \hline
    \hline
    2D $(\ell = \varpi)$& $h$ & Error$(\sigma)$ & Order & Error$(\omega)$  & Order & Error$(u)$  & Order & Error$(r)$  & Order & DoF \\
    \hline
    \multirow{4}{*}{\rotatebox[origin=c]{90}{MFE}}
& 7.85e-02 & 6.38e-03 &    - & 2.99e-02 &    - & 1.57e-01 &    - & 1.12e-01 &    - & 7.72e+03 \\ 
& 4.25e-02 & 2.75e-03 & 1.37 & 1.55e-02 & 1.07 & 7.98e-02 & 1.10 & 5.64e-02 & 1.12 & 2.94e+04 \\ 
& 2.05e-02 & 1.30e-03 & 1.03 & 7.62e-03 & 0.97 & 3.97e-02 & 0.96 & 2.80e-02 & 0.96 & 1.19e+05 \\ 
& 1.04e-02 & 6.46e-04 & 1.03 & 3.87e-03 & 0.99 & 2.00e-02 & 1.01 & 1.41e-02 & 1.01 & 4.65e+05 \\
    \hline
    \multirow{4}{*}{\rotatebox[origin=c]{90}{MS-MFE}}
& 7.85e-02 & 2.44e-02 &    - & 4.52e-02 &    - & 1.58e-01 &    - & 1.21e-01 &    - & 1.88e+03 \\ 
& 4.25e-02 & 1.22e-02 & 1.13 & 2.28e-02 & 1.12 & 7.99e-02 & 1.11 & 6.13e-02 & 1.11 & 7.25e+03 \\ 
& 2.05e-02 & 6.02e-03 & 0.97 & 1.11e-02 & 0.98 & 3.97e-02 & 0.96 & 3.04e-02 & 0.96 & 2.95e+04 \\ 
& 1.04e-02 & 3.00e-03 & 1.02 & 5.62e-03 & 1.00 & 2.00e-02 & 1.01 & 1.54e-02 & 1.00 & 1.16e+05 \\
    \hline
%
%
    \hline
    3D $(\ell = 1)$ & $h$ & Error$(\sigma)$ & Order & Error$(\omega)$  & Order & Error$(u)$  & Order & Error$(r)$  & Order & DoF \\
    \hline
    \multirow{4}{*}{\rotatebox[origin=c]{90}{MFE}}
& 5.77e-01 & 2.64e-01 &    - & 2.27e-01 &    - & 1.29e+00 &    - & 1.29e+00 &    - & 7.78e+03 \\ 
& 2.89e-01 & 6.24e-02 & 2.08 & 5.35e-02 & 2.08 & 6.98e-01 & 0.89 & 6.97e-01 & 0.89 & 5.83e+04 \\ 
& 1.92e-01 & 2.87e-02 & 1.91 & 2.33e-02 & 2.05 & 4.73e-01 & 0.96 & 4.73e-01 & 0.96 & 1.92e+05 \\ 
& 1.44e-01 & 1.72e-02 & 1.78 & 1.30e-02 & 2.03 & 3.57e-01 & 0.98 & 3.57e-01 & 0.98 & 4.51e+05 \\
    \hline
    \multirow{4}{*}{\rotatebox[origin=c]{90}{MS-MFE}}
& 5.77e-01 & 3.94e-01 &    - & 3.53e-01 &    - & 1.41e+00 &    - & 1.40e+00 &    - & 9.72e+02 \\ 
& 2.89e-01 & 1.44e-01 & 1.45 & 1.39e-01 & 1.35 & 7.15e-01 & 0.98 & 7.13e-01 & 0.98 & 7.78e+03 \\ 
& 1.92e-01 & 8.84e-02 & 1.20 & 8.67e-02 & 1.16 & 4.79e-01 & 0.99 & 4.77e-01 & 0.99 & 2.62e+04 \\ 
& 1.44e-01 & 6.43e-02 & 1.11 & 6.35e-02 & 1.08 & 3.59e-01 & 1.00 & 3.59e-01 & 0.99 & 6.22e+04 \\
    \hline
    \hline
    3D $(\ell = \varpi)$& $h$ & Error$(\sigma)$ & Order & Error$(\omega)$  & Order & Error$(u)$  & Order & Error$(r)$  & Order & DoF \\
    \hline
    \multirow{4}{*}{\rotatebox[origin=c]{90}{MFE}}
& 5.77e-01 & 2.64e-01 &    - & 3.04e-01 &    - & 1.29e+00 &    - & 1.32e+00 &    - & 7.78e+03 \\ 
& 2.89e-01 & 6.33e-02 & 2.06 & 1.31e-01 & 1.21 & 6.98e-01 & 0.89 & 7.23e-01 & 0.87 & 5.83e+04 \\ 
& 1.92e-01 & 2.92e-02 & 1.91 & 6.39e-02 & 1.77 & 4.73e-01 & 0.96 & 4.83e-01 & 1.00 & 1.92e+05 \\ 
& 1.44e-01 & 1.76e-02 & 1.76 & 4.10e-02 & 1.54 & 3.57e-01 & 0.98 & 3.62e-01 & 1.00 & 4.51e+05 \\
    \hline
    \multirow{4}{*}{\rotatebox[origin=c]{90}{MS-MFE}}
& 5.77e-01 & 3.96e-01 &    - & 4.30e-01 &    - & 1.41e+00 &    - & 1.42e+00 &    - & 9.72e+02 \\ 
& 2.89e-01 & 1.45e-01 & 1.45 & 2.07e-01 & 1.06 & 7.15e-01 & 0.98 & 7.55e-01 & 0.91 & 7.78e+03 \\ 
& 1.92e-01 & 8.91e-02 & 1.20 & 1.19e-01 & 1.36 & 4.79e-01 & 0.99 & 4.98e-01 & 1.03 & 2.62e+04 \\ 
& 1.44e-01 & 6.47e-02 & 1.11 & 8.44e-02 & 1.20 & 3.59e-01 & 1.00 & 3.71e-01 & 1.02 & 6.22e+04 \\
    \hline
    \end{tabular}
    }
\end{table}

We first consider the lowest-order methods proposed in \Cref{sec:simple_scheme}. As shown in \Cref{tab: BDM1-P0-2d}, these schemes converge linearly with respect to the mesh size. It is notable that the multipoint stress method achieves the same error in the displacement and rotation variables as the full mixed finite element method, with significantly fewer degrees of freedom. In particular, for these grid families, we notice that in 2D, the size of the corresponding linear system for MS-MFE is about $25\%$ of the size of MFE and $14\%$ in 3D.

We observe certain superlinear convergence behavior for the MFE method, namely in $\omega$ in 2D and in both stress variables in 3D. This was similarly observed in the numerical experiments of \cite[Sec.~5]{boon2025mixed}. However, the introduction of the quadrature rule eliminates this behavior, leading to first order convergence for the MS-MFE method, in agreement with \Cref{thm:conv MS BDM1-P0}.

\subsection{The reducible scheme \texorpdfstring{$\BDM1$-$\L1$}{BDM1-L1}}
\label{sub:the_reducible_scheme}
\begin{table}[ht]
    \centering
    \caption{Convergence results for the $\BDM1$-$\L1$ mixed finite element methods from \Cref{sec:reducible_scheme}.}
    \label{tab: BDM1-L1-2d}
    {\footnotesize
    \begin{tabular}{|c|c|cc|cc|cc|cc|c|}
    \hline
    2D $(\ell = 1)$ &$h$ & Error$(\sigma)$ & Order & Error$(\omega)$  & Order & Error$(u)$  & Order & Error$(r)$  & Order & DoF \\
    \hline
    \multirow{4}{*}{\rotatebox[origin=c]{90}{MFE}}
& 7.85e-02 & 1.59e-02 &    - & 1.14e-02 &    - & 1.57e-01 &    - & 4.62e-03 &    - & 7.44e+03 \\ 
& 4.25e-02 & 7.79e-03 & 1.17 & 3.73e-03 & 1.82 & 7.98e-02 & 1.10 & 1.21e-03 & 2.18 & 2.83e+04 \\ 
& 2.05e-02 & 3.84e-03 & 0.97 & 1.34e-03 & 1.41 & 3.97e-02 & 0.96 & 3.00e-04 & 1.91 & 1.14e+05 \\ 
& 1.04e-02 & 1.93e-03 & 1.01 & 4.10e-04 & 1.74 & 2.00e-02 & 1.01 & 7.62e-05 & 2.02 & 4.45e+05 \\
    \hline
    \multirow{4}{*}{\rotatebox[origin=c]{90}{MS-MFE}}
& 7.85e-02 & 2.38e-02 &    - & 3.47e-02 &    - & 1.58e-01 &    - & 1.15e-02 &    - & 1.60e+03 \\ 
& 4.25e-02 & 1.20e-02 & 1.12 & 1.73e-02 & 1.14 & 7.99e-02 & 1.11 & 2.88e-03 & 2.25 & 6.11e+03 \\ 
& 2.05e-02 & 5.90e-03 & 0.97 & 8.46e-03 & 0.98 & 3.97e-02 & 0.96 & 7.10e-04 & 1.92 & 2.47e+04 \\ 
& 1.04e-02 & 2.95e-03 & 1.02 & 4.25e-03 & 1.01 & 2.00e-02 & 1.01 & 1.80e-04 & 2.02 & 9.67e+04 \\
    \hline
    \hline
    2D $(\ell=\varpi)$ &$h$ & Error$(\sigma)$ & Order & Error$(\omega)$  & Order & Error$(u)$  & Order & Error$(r)$  & Order & DoF \\
    \hline
    \multirow{4}{*}{\rotatebox[origin=c]{90}{MFE}}
& 7.85e-02 & 1.51e-02 &    - & 8.62e-02 &    - & 1.57e-01 &    - & 3.57e-02 &    - & 7.44e+03 \\ 
& 4.25e-02 & 7.59e-03 & 1.12 & 4.62e-02 & 1.02 & 7.98e-02 & 1.10 & 1.28e-02 & 1.67 & 2.83e+04 \\ 
& 2.05e-02 & 3.80e-03 & 0.95 & 2.29e-02 & 0.96 & 3.97e-02 & 0.96 & 4.50e-03 & 1.44 & 1.14e+05 \\ 
& 1.04e-02 & 1.92e-03 & 1.00 & 1.16e-02 & 1.00 & 2.00e-02 & 1.01 & 1.58e-03 & 1.54 & 4.45e+05 \\
    \hline
    \multirow{4}{*}{\rotatebox[origin=c]{90}{MS-MFE}}
& 7.85e-02 & 2.38e-02 &    - & 5.32e-02 &    - & 1.58e-01 &    - & 1.92e-02 &    - & 1.60e+03 \\ 
& 4.25e-02 & 1.20e-02 & 1.12 & 2.72e-02 & 1.09 & 7.99e-02 & 1.11 & 5.94e-03 & 1.91 & 6.11e+03 \\ 
& 2.05e-02 & 5.90e-03 & 0.97 & 1.34e-02 & 0.98 & 3.97e-02 & 0.96 & 1.99e-03 & 1.50 & 2.47e+04 \\ 
& 1.04e-02 & 2.95e-03 & 1.02 & 6.77e-03 & 1.00 & 2.00e-02 & 1.01 & 6.75e-04 & 1.59 & 9.67e+04 \\
    \hline
%
    \hline
    3D $(\ell = 1)$ &$h$ & Error$(\sigma)$ & Order & Error$(\omega)$  & Order & Error$(u)$  & Order & Error$(r)$  & Order & DoF \\
    \hline
    \multirow{4}{*}{\rotatebox[origin=c]{90}{MFE}}
& 5.77e-01 & 4.54e-01 &    - & 5.22e-01 &    - & 1.31e+00 &    - & 4.97e-01 &    - & 7.48e+03 \\ 
& 2.89e-01 & 1.08e-01 & 2.07 & 1.36e-01 & 1.94 & 6.99e-01 & 0.90 & 5.81e-02 & 3.10 & 5.55e+04 \\ 
& 1.92e-01 & 6.06e-02 & 1.43 & 6.77e-02 & 1.71 & 4.74e-01 & 0.96 & 2.21e-02 & 2.38 & 1.82e+05 \\ 
& 1.44e-01 & 4.24e-02 & 1.24 & 4.34e-02 & 1.55 & 3.57e-01 & 0.98 & 1.24e-02 & 2.01 & 4.26e+05 \\
    \hline
    \multirow{4}{*}{\rotatebox[origin=c]{90}{MS-MFE}}
& 5.77e-01 & 4.83e-01 &    - & 6.41e-01 &    - & 1.44e+00 &    - & 2.15e+00 &    - & 6.78e+02 \\ 
& 2.89e-01 & 1.44e-01 & 1.75 & 1.76e-01 & 1.86 & 7.15e-01 & 1.01 & 1.73e-01 & 3.64 & 4.92e+03 \\ 
& 1.92e-01 & 8.75e-02 & 1.23 & 9.75e-02 & 1.46 & 4.79e-01 & 0.99 & 4.79e-02 & 3.16 & 1.61e+04 \\ 
& 1.44e-01 & 6.36e-02 & 1.11 & 6.78e-02 & 1.26 & 3.59e-01 & 1.00 & 2.76e-02 & 1.92 & 3.77e+04 \\
    \hline
    \hline
    3D $(\ell=\varpi)$ &$h$ & Error$(\sigma)$ & Order & Error$(\omega)$  & Order & Error$(u)$  & Order & Error$(r)$  & Order & DoF \\
    \hline
    \multirow{4}{*}{\rotatebox[origin=c]{90}{MFE}}
& 5.77e-01 & 3.10e-01 &    - & 6.10e-01 &    - & 1.29e+00 &    - & 4.49e-01 &    - & 7.48e+03 \\ 
& 2.89e-01 & 8.91e-02 & 1.80 & 2.67e-01 & 1.19 & 6.99e-01 & 0.89 & 2.80e-01 & 0.68 & 5.55e+04 \\ 
& 1.92e-01 & 5.27e-02 & 1.29 & 1.75e-01 & 1.04 & 4.74e-01 & 0.96 & 1.79e-01 & 1.10 & 1.82e+05 \\ 
& 1.44e-01 & 3.84e-02 & 1.11 & 1.32e-01 & 0.99 & 3.57e-01 & 0.98 & 1.27e-01 & 1.20 & 4.26e+05 \\
    \hline
    \multirow{4}{*}{\rotatebox[origin=c]{90}{MS-MFE}}
& 5.77e-01 & 4.04e-01 &    - & 7.08e-01 &    - & 1.41e+00 &    - & 7.87e-01 &    - & 6.78e+02 \\ 
& 2.89e-01 & 1.43e-01 & 1.50 & 2.32e-01 & 1.61 & 7.15e-01 & 0.98 & 1.60e-01 & 2.30 & 4.92e+03 \\ 
& 1.92e-01 & 8.74e-02 & 1.21 & 1.31e-01 & 1.41 & 4.78e-01 & 0.99 & 7.61e-02 & 1.83 & 1.61e+04 \\ 
& 1.44e-01 & 6.35e-02 & 1.11 & 9.08e-02 & 1.28 & 3.59e-01 & 0.99 & 4.88e-02 & 1.54 & 3.77e+04 \\
    \hline
    \end{tabular}
    }
\end{table}

The results for the reducible scheme of \Cref{sec:reducible_scheme} are shown in \Cref{tab: BDM1-L1-2d}. Again, we observe first order convergence, at least, in all variables, as predicted by \Cref{thm:convergence BDM1-L1} and \Cref{thm:conv MSMFE BDML}. The rotation variable appears to converge superlinearly in some cases, even after the introduction of the quadrature rule. We also observe quadratic convergence in this variable if $\ell = 1$. It may therefore be possible to improve our analysis. 

Also in this case the reduction in the linear system size is significant, in 2D MS-MFE is about $22\%$ of the size of MFE and $9\%$ in 3D. In these tests, we have not performed the additional reduction of the rotation variable in the region where $\ell = 0$. This is only an algebraic manipulation and therefore does not affect the numerical solution.

\subsection{The higher-order scheme with continuous rotations \texorpdfstring{$\RT1$-$\L1$}{RT1-L1}}
\label{sub: numerics rt1_l1}
\begin{table}[ht]
    \centering
    \caption{Convergence results for the $\RT1$-$\L1$ mixed finite element methods from \Cref{sec:RT1-L1}.}
    \label{tab: RT1-L1-2d}
    {\footnotesize
    \begin{tabular}{|c|c|cc|cc|cc|cc|c|}
    \hline
    2D $(\ell = 1)$&$h$ & Error$(\sigma)$ & Order & Error$(\omega)$  & Order & Error$(u)$  & Order & Error$(r)$  & Order & DoF \\
    \hline
    \multirow{4}{*}{\rotatebox[origin=c]{90}{MFE}}
& 7.85e-02 & 4.75e-03 &    - & 3.97e-02 &    - & 1.85e-03 &    - & 2.67e-03 &    - & 1.37e+04 \\ 
& 4.25e-02 & 1.17e-03 & 2.28 & 2.05e-02 & 1.08 & 4.86e-04 & 2.18 & 6.92e-04 & 2.20 & 5.24e+04 \\ 
& 2.05e-02 & 2.86e-04 & 1.93 & 1.01e-02 & 0.97 & 1.18e-04 & 1.95 & 1.68e-04 & 1.94 & 2.12e+05 \\ 
& 1.04e-02 & 7.26e-05 & 2.02 & 5.06e-03 & 1.02 & 2.93e-05 & 2.04 & 4.20e-05 & 2.03 & 8.31e+05 \\
    \hline
    \multirow{4}{*}{\rotatebox[origin=c]{90}{MS-MFE}}
& 7.85e-02 & 4.78e-03 &    - & 3.22e-02 &    - & 2.66e-03 &    - & 2.86e-03 &    - & 4.10e+03 \\ 
& 4.25e-02 & 1.18e-03 & 2.29 & 1.63e-02 & 1.11 & 7.03e-04 & 2.17 & 7.40e-04 & 2.20 & 1.58e+04 \\ 
& 2.05e-02 & 2.87e-04 & 1.93 & 8.00e-03 & 0.98 & 1.71e-04 & 1.94 & 1.79e-04 & 1.94 & 6.40e+04 \\ 
& 1.04e-02 & 7.27e-05 & 2.02 & 3.97e-03 & 1.03 & 4.31e-05 & 2.03 & 4.49e-05 & 2.04 & 2.51e+05 \\
    \hline
    \hline
    2D $(\ell=\varpi)$&$h$ & Error$(\sigma)$ & Order & Error$(\omega)$  & Order & Error$(u)$  & Order & Error$(r)$  & Order & DoF \\
    \hline
    \multirow{4}{*}{\rotatebox[origin=c]{90}{MFE}}
& 7.85e-02 & 4.86e-03 &    - & 4.20e-02 &    - & 1.89e-03 &    - & 8.66e-03 &    - & 1.37e+04 \\ 
& 4.25e-02 & 1.20e-03 & 2.28 & 2.15e-02 & 1.09 & 4.95e-04 & 2.18 & 2.44e-03 & 2.06 & 5.24e+04 \\ 
& 2.05e-02 & 2.93e-04 & 1.93 & 1.05e-02 & 0.98 & 1.19e-04 & 1.95 & 5.42e-04 & 2.07 & 2.12e+05 \\ 
& 1.04e-02 & 7.43e-05 & 2.02 & 5.26e-03 & 1.02 & 2.97e-05 & 2.04 & 1.39e-04 & 2.00 & 8.31e+05 \\
    \hline
    \multirow{4}{*}{\rotatebox[origin=c]{90}{MS-MFE}}
& 7.85e-02 & 4.89e-03 &    - & 3.48e-02 &    - & 2.69e-03 &    - & 8.64e-03 &    - & 4.10e+03 \\ 
& 4.25e-02 & 1.20e-03 & 2.28 & 1.74e-02 & 1.13 & 7.09e-04 & 2.17 & 2.32e-03 & 2.15 & 1.58e+04 \\ 
& 2.05e-02 & 2.93e-04 & 1.94 & 8.44e-03 & 1.00 & 1.72e-04 & 1.94 & 5.09e-04 & 2.08 & 6.40e+04 \\ 
& 1.04e-02 & 7.42e-05 & 2.02 & 4.19e-03 & 1.03 & 4.34e-05 & 2.03 & 1.28e-04 & 2.03 & 2.51e+05 \\
    \hline
%
    \hline
    3D $(\ell = 1)$&$h$ & Error$(\sigma)$ & Order & Error$(\omega)$  & Order & Error$(u)$  & Order & Error$(r)$  & Order & DoF \\
    \hline
     \multirow{4}{*}{\rotatebox[origin=c]{90}{MFE}}
& 5.77e-01 & 2.61e-01 &    - & 3.24e-01 &    - & 2.05e-01 &    - & 1.91e-01 &    - & 1.19e+04 \\ 
& 2.89e-01 & 7.13e-02 & 1.87 & 1.51e-01 & 1.10 & 4.87e-02 & 2.08 & 5.69e-02 & 1.75 & 9.05e+04 \\ 
& 1.92e-01 & 3.22e-02 & 1.96 & 9.26e-02 & 1.20 & 2.13e-02 & 2.04 & 2.46e-02 & 2.07 & 3.00e+05 \\ 
& 1.44e-01 & 1.82e-02 & 1.98 & 6.65e-02 & 1.15 & 1.19e-02 & 2.02 & 1.36e-02 & 2.07 & 7.06e+05 \\
     \hline
    \multirow{4}{*}{\rotatebox[origin=c]{90}{MS-MFE}}
& 5.77e-01 & 2.62e-01 &    - & 6.66e-01 &    - & 2.46e-01 &    - & 2.98e-01 &    - & 2.14e+03 \\ 
& 2.89e-01 & 7.24e-02 & 1.85 & 4.81e-01 & 0.47 & 5.83e-02 & 2.08 & 1.36e-01 & 1.13 & 1.66e+04 \\ 
& 1.92e-01 & 3.28e-02 & 1.95 & 3.54e-01 & 0.76 & 2.56e-02 & 2.03 & 6.83e-02 & 1.69 & 5.55e+04 \\ 
& 1.44e-01 & 1.86e-02 & 1.98 & 2.75e-01 & 0.87 & 1.43e-02 & 2.02 & 4.01e-02 & 1.85 & 1.31e+05 \\
    \hline
    \hline
    3D $(\ell=\varpi)$&$h$ & Error$(\sigma)$ & Order & Error$(\omega)$  & Order & Error$(u)$  & Order & Error$(r)$  & Order & DoF \\
    \hline
    \multirow{4}{*}{\rotatebox[origin=c]{90}{MFE}}
& 5.77e-01 & 2.63e-01 &    - & 3.59e-01 &    - & 2.05e-01 &    - & 3.09e-01 &    - & 1.19e+04 \\ 
& 2.89e-01 & 7.21e-02 & 1.87 & 1.84e-01 & 0.97 & 4.86e-02 & 2.08 & 7.93e-02 & 1.96 & 9.05e+04 \\ 
& 1.92e-01 & 3.26e-02 & 1.95 & 1.04e-01 & 1.39 & 2.13e-02 & 2.04 & 4.22e-02 & 1.56 & 3.00e+05 \\ 
& 1.44e-01 & 1.85e-02 & 1.97 & 7.21e-02 & 1.29 & 1.19e-02 & 2.02 & 2.81e-02 & 1.41 & 7.06e+05 \\
    \hline
    \multirow{4}{*}{\rotatebox[origin=c]{90}{MS-MFE}}
& 5.77e-01 & 2.63e-01 &    - & 6.46e-01 &    - & 2.47e-01 &    - & 3.50e-01 &    - & 2.14e+03 \\ 
& 2.89e-01 & 7.23e-02 & 1.86 & 4.95e-01 & 0.38 & 5.83e-02 & 2.09 & 9.59e-02 & 1.87 & 1.66e+04 \\ 
& 1.92e-01 & 3.28e-02 & 1.95 & 3.56e-01 & 0.81 & 2.55e-02 & 2.04 & 5.13e-02 & 1.54 & 5.55e+04 \\ 
& 1.44e-01 & 1.86e-02 & 1.96 & 2.77e-01 & 0.88 & 1.43e-02 & 2.02 & 3.43e-02 & 1.40 & 1.31e+05 \\
    \hline
    \end{tabular}
    }
\end{table}

\Cref{tab: RT1-L1-2d} present the convergence results for the $\RT1$-$\L1$ schemes proposed in \Cref{sec:RT1-L1}. We showed in \Cref{thm:conv full RT1-L1} and \Cref{thm:stab/conv MSMFE RT1L1} that these methods would converge with only first order if $\ell \ne 0$. However, we nevertheless observe superlinear convergence in the Cauchy stress, displacement, and rotation variables.
These results appear to indicate that the loss in convergence highlighted in \eqref{eq: bound problematic term L1} only affects the couple stress $\omega$. 

In the 3D test case, the couple stress converges the slowest, and is the only variable that does not exhibit a convincing linear convergence. This may be due to the fact that the grids in 3D are too coarse to illustrate the asymptotic behavior. However, the coarse grids were necessary to keep the number of degrees of freedom of these methods manageable by the linear solver. We moreover note that while the rotation converges quadratically in 2D, its rate is reduced in 3D if the quadrature rule is introduced or $\ell$ is spatially varying. Finally, the reduction in the numbers of degrees of freedom is similar to the previous tests; the system of MS-MFE is about $30\%$ of the size of MFE in 2D and $19\%$ in 3D.

\subsection{The higher-order scheme with discontinuous rotations \texorpdfstring{$\RT1$-$\P1$}{RT1-P1}}
\label{sub: numerics rt1_p1}
\begin{table}[ht]
    \centering
    \caption{Convergence results for the $\RT1$-$\P1$ mixed finite element methods from \Cref{sec:RT1-P1}.}
    \label{tab: RT1-P1-2d}
    {\footnotesize
    \begin{tabular}{|c|c|cc|cc|cc|cc|c|}
    \hline
    2D $(\ell = 1)$&$h$ & Error$(\sigma)$ & Order & Error$(\omega)$  & Order & Error$(u)$  & Order & Error$(r)$  & Order & DoF \\
    \hline
    \multirow{4}{*}{\rotatebox[origin=c]{90}{MFE}}
& 7.85e-02 & 3.35e-03 &    - & 3.07e-03 &    - & 1.85e-03 &    - & 2.94e-03 &    - & 4.53e+04 \\ 
& 4.25e-02 & 8.41e-04 & 2.25 & 7.57e-04 & 2.28 & 4.86e-04 & 2.18 & 7.17e-04 & 2.30 & 1.74e+05 \\ 
& 2.05e-02 & 2.08e-04 & 1.92 & 1.86e-04 & 1.92 & 1.19e-04 & 1.93 & 1.76e-04 & 1.93 & 7.08e+05 \\ 
& 1.04e-02 & 5.26e-05 & 2.02 & 4.69e-05 & 2.02 & 3.00e-05 & 2.02 & 4.43e-05 & 2.02 & 2.78e+06 \\
    \hline
    \multirow{4}{*}{\rotatebox[origin=c]{90}{MS-MFE}}
& 7.85e-02 & 3.39e-03 &    - & 3.10e-03 &    - & 2.23e-03 &    - & 2.99e-03 &    - & 1.69e+04 \\ 
& 4.25e-02 & 8.50e-04 & 2.25 & 7.65e-04 & 2.28 & 5.82e-04 & 2.19 & 7.32e-04 & 2.30 & 6.53e+04 \\ 
& 2.05e-02 & 2.10e-04 & 1.91 & 1.88e-04 & 1.93 & 1.42e-04 & 1.93 & 1.79e-04 & 1.93 & 2.65e+05 \\ 
& 1.04e-02 & 5.34e-05 & 2.02 & 4.74e-05 & 2.02 & 3.59e-05 & 2.02 & 4.52e-05 & 2.02 & 1.04e+06 \\
    \hline
    \hline
    2D $(\ell = \varpi)$&$h$ & Error$(\sigma)$ & Order & Error$(\omega)$  & Order & Error$(u)$  & Order & Error$(r)$  & Order & DoF \\
    \hline
    \multirow{4}{*}{\rotatebox[origin=c]{90}{MFE}}
& 7.85e-02 & 3.57e-03 &    - & 1.09e-02 &    - & 1.86e-03 &    - & 6.86e-03 &    - & 4.53e+04 \\ 
& 4.25e-02 & 9.11e-04 & 2.23 & 2.80e-03 & 2.21 & 4.87e-04 & 2.19 & 2.37e-03 & 1.73 & 1.74e+05 \\ 
& 2.05e-02 & 2.26e-04 & 1.91 & 6.63e-04 & 1.97 & 1.19e-04 & 1.93 & 6.04e-04 & 1.88 & 7.08e+05 \\ 
& 1.04e-02 & 5.71e-05 & 2.02 & 1.69e-04 & 2.01 & 3.00e-05 & 2.02 & 1.43e-04 & 2.12 & 2.78e+06 \\
    \hline
    \multirow{4}{*}{\rotatebox[origin=c]{90}{MS-MFE}}
& 7.85e-02 & 3.60e-03 &    - & 1.09e-02 &    - & 2.24e-03 &    - & 7.00e-03 &    - & 1.69e+04 \\ 
& 4.25e-02 & 9.19e-04 & 2.22 & 2.81e-03 & 2.21 & 5.83e-04 & 2.20 & 2.45e-03 & 1.71 & 6.53e+04 \\ 
& 2.05e-02 & 2.29e-04 & 1.91 & 6.65e-04 & 1.98 & 1.42e-04 & 1.93 & 6.28e-04 & 1.87 & 2.65e+05 \\ 
& 1.04e-02 & 5.79e-05 & 2.02 & 1.70e-04 & 2.00 & 3.60e-05 & 2.02 & 1.52e-04 & 2.08 & 1.04e+06 \\
    \hline
%
    \hline
    3D $(\ell = 1)$&$h$ & Error$(\sigma)$ & Order & Error$(\omega)$  & Order & Error$(u)$  & Order & Error$(r)$  & Order & DoF \\
    \hline
     \multirow{3}{*}{\rotatebox[origin=c]{90}{MFE}}
& 5.77e-01 & 2.91e-01 &    - & 2.73e-01 &    - & 1.65e-01 &    - & 1.25e-01 &    - & 5.15e+04 \\ 
& 2.89e-01 & 7.96e-02 & 1.87 & 7.42e-02 & 1.88 & 4.12e-02 & 2.01 & 3.48e-02 & 1.85 & 4.08e+05 \\ 
& 1.92e-01 & 3.58e-02 & 1.95 & 3.34e-02 & 1.95 & 1.82e-02 & 2.00 & 1.58e-02 & 1.93 & 1.37e+06 \\
     \hline
    \multirow{4}{*}{\rotatebox[origin=c]{90}{MS-MFE}}
& 5.77e-01 & 2.90e-01 &    - & 2.72e-01 &    - & 1.99e-01 &    - & 1.33e-01 &    - & 1.56e+04 \\ 
& 2.89e-01 & 7.92e-02 & 1.87 & 7.40e-02 & 1.88 & 4.88e-02 & 2.03 & 3.63e-02 & 1.87 & 1.24e+05 \\ 
& 1.92e-01 & 3.56e-02 & 1.97 & 3.33e-02 & 1.97 & 2.16e-02 & 2.01 & 1.65e-02 & 1.95 & 4.20e+05 \\ 
& 1.44e-01 & 2.01e-02 & 1.99 & 1.88e-02 & 1.99 & 1.21e-02 & 2.01 & 9.35e-03 & 1.97 & 9.95e+05 \\
    \hline
    \hline
    3D $(\ell = \varpi)$&$h$ & Error$(\sigma)$ & Order & Error$(\omega)$  & Order & Error$(u)$  & Order & Error$(r)$  & Order & DoF \\
    \hline
     \multirow{3}{*}{\rotatebox[origin=c]{90}{MFE}}
& 5.77e-01 & 2.95e-01 &    - & 3.85e-01 &    - & 1.67e-01 &    - & 3.79e-01 &    - & 5.15e+04 \\ 
& 2.89e-01 & 8.17e-02 & 1.86 & 1.81e-01 & 1.09 & 4.17e-02 & 2.01 & 1.40e-01 & 1.44 & 4.08e+05 \\ 
& 1.92e-01 & 3.69e-02 & 1.94 & 8.21e-02 & 1.93 & 1.84e-02 & 2.00 & 6.94e-02 & 1.72 & 1.37e+06 \\
     \hline
    \multirow{4}{*}{\rotatebox[origin=c]{90}{MS-MFE}}
& 5.77e-01 & 2.93e-01 &    - & 3.87e-01 &    - & 2.03e-01 &    - & 3.77e-01 &    - & 1.56e+04 \\ 
& 2.89e-01 & 8.12e-02 & 1.85 & 1.81e-01 & 1.10 & 4.93e-02 & 2.04 & 1.40e-01 & 1.43 & 1.24e+05 \\ 
& 1.92e-01 & 3.67e-02 & 1.96 & 8.21e-02 & 1.95 & 2.18e-02 & 2.02 & 6.89e-02 & 1.75 & 4.20e+05 \\ 
& 1.44e-01 & 2.08e-02 & 1.97 & 4.69e-02 & 1.94 & 1.22e-02 & 2.01 & 4.29e-02 & 1.64 & 9.95e+05 \\
    \hline
    \end{tabular}
    }
\end{table}

\Cref{tab: RT1-P1-2d} show the behavior of the $\RT1$-$\P1$ scheme for both the full and the multipoint stress mixed finite element methods. Recall that for these methods, we have to use barycentrically subdivided grids. The results confirm the quadratic rates predicted for the MFE method in \Cref{thm: stab/conv full RT1-P1} and the MS-MFE method in \Cref{thm:stab/conv MSMFE RT1P1}. The only exception is the rotation for MS-MFE in 3D with spatially varying $\ell$, for which the rate is slightly less then two, but appear to be approaching two as the grids are refined. In addition, it is notable that the couple stress and the displacement, for which only first order convergence is established for the MS-MFE method in \Cref{thm:stab/conv MSMFE RT1P1}, converge quadratically also after applying the quadrature rule.

Note that the barycentric subdivision of the grid does not impact the mesh size $h$ since the diameters of the simplices is unaffected. However, the subdivision significantly increases the number of degrees of freedom, cf.~\Cref{rem:barycentric}. Due to the larger system size, it was not computationally feasible to obtain the MFE results on the finest grid in 3D. Finally, in terms of the numbers of degrees of freedom, we report that MS-MFE is about $37\%$ of the size of MFE in 2D, and $30\%$ in 3D.

\section{Conclusion}
\label{sec:conclusion}

We have proposed and analyzed four multipoint stress mixed finite element methods for the linearized Cosserat equations. These methods were characterized by a low-order quadrature rule with which the Cauchy and couple stress variables can be eliminated locally. The numerical schemes therefore only contain the displacement and rotation variables. Through a priori error estimates, we showed that each of the variants converges linearly or quadratically if the exact solution is sufficiently regular. Numerical experiments support these analytical results, and we moreover observed higher convergence rates than expected in some variables.

\section*{Acknowledgments}

The second author has been partially funded by the PRIN project ``FREYA - Fault REactivation: a hYbrid numerical Approach'' - SLG2RIST01. The present research is part of the activities of ``Dipartimento di Eccellenza 2023-2027'', Italian Minister of University and Research (MUR), grant Dipartimento di Eccellenza 2023-2027, under the project ``PON Ricerca e Innovazione 2014-2020". The fourth author has been partially supported by the US National Science Foundation through grant DMS-2410686.

\bibliographystyle{elsarticle-num}
\bibliography{references}

\appendix
\renewcommand{\appendixname}{}

\end{document}